\documentclass{amsart}
\usepackage{amsmath, amssymb,amscd, amsthm, tikz-cd, mathrsfs}
\usepackage[margin=1in]{geometry}

\usepackage{hyperref, calligra}

\newcommand{\E}{\text{\calligra xt}\,}

\newcommand{\Addresses}{{
  \bigskip
  \footnotesize

  \textsc{Department of Mathematics, University of Notre Dame,
    Notre Dame, IN 46556}\par\nopagebreak
  \textit{E-mail address}: \texttt{mperlman@nd.edu}

}}

\title{equivariant $\mathcal{D}$-modules on $2\times 2\times 2$ hypermatrices}
\author{Michael Perlman}

 \newtheorem{theorem}{Theorem}[section]
 \newtheorem{lemma}[theorem]{Lemma}
 
 \newtheorem{prop}[theorem]{Proposition}

\newtheorem* {class}{Classification of Simple Modules}
\newtheorem* {quiv}{Theorem on the Quiver Structure}
\newtheorem {convention}[theorem]{Convention}

\numberwithin{equation}{section}

\newcommand{\C}{\mathbb{C}}
\newcommand{\D}{\mathcal{D}}
\newcommand{\tn}{\textnormal}
\newcommand{\bS}{\mathbb{S}}
\newcommand{\defi}[1]{{\upshape\sffamily #1}}

\begin{document}
\begin{abstract}
Let $V=\mathbb{C}^2\otimes \mathbb{C}^2\otimes \mathbb{C}^2$ be the space of $2\times 2\times 2$ hypermatrices, endowed with the natural group action of $\tn{GL}=\textnormal{GL}_2(\C)\times \textnormal{GL}_2(\C)\times \textnormal{GL}_2(\C)$. The category of $\tn{GL}$-equivariant coherent left $\D$-modules on $V$ is equivalent to the category of representations of a quiver with relations. In this article, we give a construction of each simple object and study their $\tn{GL}$-equivariant structure. Using this information, we go on to explicitly describe the corresponding quiver with relations. As an application, we compute all iterations of local cohomology with support in the orbit closures of $V$.
\end{abstract}
\maketitle

\section{Introduction}
Let $A$, $B$, $C$ be two-dimensional complex vector spaces, and let $V=A\otimes B\otimes C$ be the space of $2\times 2\times 2$ hypermatrices. This space has a natural action of $\tn{GL}=\textnormal{GL}(A)\times \textnormal{GL}(B)\times \textnormal{GL}(C)$ with seven orbits. We describe these orbits below (see \cite[Table 10.3.1]{landsberg2012tensors}), choosing bases $A=\langle a_1,a_2\rangle$, $B=\langle b_1,b_2\rangle$, $C=\langle c_1,c_2\rangle$. 
\begin{itemize}
\item The zero orbit $O_0=\{0\}$.

\item The orbit $O_1$ of dimension $4$, with representative $a_1\otimes b_1\otimes c_1$, whose closure $\overline{O_1}$ is the affine cone over the Segre variety $\textnormal{Seg}(\mathbb{P}(A)\times \mathbb{P}(B)\times \mathbb{P}(C))\subseteq \mathbb{P}(V)$.

\item The orbit $O_{1,2,2}$ of dimension $5$, with representative $a_1\otimes (b_1\otimes c_1+ b_2\otimes c_2)$, whose closure $\overline{O_{1,2,2}}$ is the subspace variety
$$
\textnormal{Sub}_{1,2,2}(V)=\{T\in V\mid \exists A'\in \mathbb{P}(A), T\in A'\otimes B\otimes C \}.
$$
The orbits $O_{2,1,2}$ and $O_{2,2,1}$ are defined similarly.

\item The orbit $O_5$ of dimension $7$, with representative $a_1\otimes(b_1\otimes c_1+b_2\otimes c_2)+a_2\otimes b_1\otimes c_2$, whose closure $\overline{O_5}$ is the affine cone over the tangential variety to the Segre variety.

\item The dense orbit $O_6$, with representative $a_1\otimes b_1\otimes c_1+a_2\otimes b_2\otimes c_2$.
\end{itemize}
Let $S=\textnormal{Sym}(V^{\ast})\cong \mathbb{C}[x_{i,j,k}\mid 1\leq i,j,k \leq 2]$ be the ring of polynomial functions on $V$, and let $\D$ be the Weyl algebra of differential operators on $V$ with polynomial coefficients. In this article we study the category $\textnormal{mod}_{\tn{GL}}(\D)$ of $\tn{GL}$-equivariant coherent left $\D$-modules. It is known from the general theory that this category is equivalent to the category of representations of a quiver with relations \cite{gelfand1996, lHorincz2018categories, vilonen1994perverse}, and our analysis entails: (1) determining the quiver with relations $(\mathcal{Q},\mathcal{I})$ corresponding to $\textnormal{mod}_{\tn{GL}}(\D)$, (2) giving concrete constructions of the simple objects and understanding their equivariant structure. The information of (1) and (2) is useful because it aids in determining the filtration and composition factors of any object in $\tn{mod}_{\tn{GL}}(\D)$ that one may come across, such as local cohomology $\mathcal{H}^{\bullet}_{\overline{O}}(V,\mathcal{O}_V)$ with support in an orbit closure. 

This work is part of an ongoing effort to understand categories of equivariant $\D$-modules on irreducible representations with finitely many orbits, such as spaces of matrices, affine cones over Veronese varieties, and more \cite{lHorincz2017equivariant, lHorincz2018categories, raicu2016characters, raicu2017characters, lp2018equivariant}. The space of $2\times 2\times 2$ hypermatrices is part of the subexceptional series of representations of finitely many orbits \cite[Section 6]{series}, which also includes the space of binary cubic forms and the space of alternating senary 3-tensors. The categories of equivariant $\D$-modules on these spaces have been studied in \cite{lHorincz2017equivariant} and \cite{lp2018equivariant} respectively.

We begin by examining the simple objects in $\textnormal{mod}_{\tn{GL}}(\D)$. By the Riemann-Hilbert correspondence, each simple corresponds to a $\tn{GL}$-equivariant local system on one of the orbits. The classification in the case of $2\times 2\times 2$ hypermatrices is as follows:

\begin{class}
There are eight simple $\tn{GL}=\textnormal{GL}(A)\times \textnormal{GL}(B)\times \textnormal{GL}(C)$-equivariant $\D$-modules on $V=A\otimes B\otimes C$. For all orbits $O\neq O_6$, there is a unique simple with support $\overline{O}$. These modules correspond to the trivial local systems on their respective orbits, and we denote them by $D_0=E$, $D_1$, $D_{1,2,2}$, $D_{2,1,2}$, $D_{2,2,1}$, and $D_5$. There are two simple objects with full support: $D_6=S$ and $G_6$. 

The holonomic duality functor fixes all of the simple modules. The Fourier transform swaps the modules in the two pairs $(S,E)$, $(G_6,D_1)$, and all other simples are fixed. 
\end{class}

We recall the definitions of holonomic duality functor and the Fourier transform in Section \ref{D}. We now state the theorem on the quiver structure of the category of $\tn{GL}$-equivariant coherent $\D$-modules.

\begin{quiv}
There is an equivalence of categories $\textnormal{mod}_{\textnormal{GL}}(\mathcal{D}_V)\cong \textnormal{rep}(\mathcal{Q},\mathcal{I})$,
where $\textnormal{rep}(\mathcal{Q},\mathcal{I})$ is the category of finite-dimensional representations of a quiver $\mathcal{Q}$ with relations $\mathcal{I}$. The quiver $\mathcal{Q}$ is shown below. \\
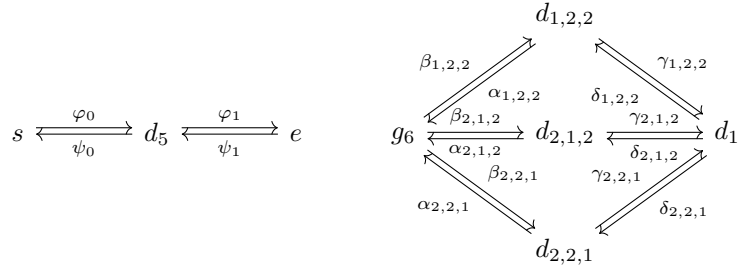
$$
\begin{tikzcd}[column sep=large, row sep=large]
& &\\
s \arrow[r, shift left, "\varphi_0"]
& d_5 \arrow[r, shift left, "\varphi_1"]
\arrow[l, "\psi_0"]
& e \arrow[l, "\psi_1"]\\
\end{tikzcd}\;\;\;\;\;\;\;\;\;
\begin{tikzcd}[column sep=large, row sep=large]
& d_{1,2,2}\arrow[rd, "\gamma_{1,2,2}"] \arrow[dl, shift left, "\alpha_{1,2,2}"] & \\
g_6 \arrow[r, "\beta_{2,1,2}"]\arrow[ur, "\beta_{1,2,2}"]\arrow[dr, "\beta_{2,2,1}"] & d_{2,1,2}\arrow[r, "\gamma_{2,1,2}"] \arrow[l, shift left, "\alpha_{2,1,2}"]& d_1\arrow[l, shift left, "\delta_{2,1,2}"] \arrow[ul, shift left, "\delta_{1,2,2}"]\arrow[dl, shift left, "\delta_{2,2,1}"]\\
& d_{2,2,1}\arrow[ur, "\gamma_{2,2,1}"]\arrow[ul, shift left, "\alpha_{2,2,1}"]
\end{tikzcd}
$$
\vspace{.5cm}

\noindent The relations on the first connected component are given by: $\varphi_0\psi_0,\psi_0\varphi_0,\varphi_1\psi_1,\psi_1\varphi_1$. The relations on the second connected component are, for all $(i,j,k)$ and $(p,q,r)$:
$$
\;\;\beta_{i,j,k}\alpha_{i,j,k},\;\; \alpha_{i,j,k}\beta_{i,j,k},\;\;\delta_{i,j,k}\gamma_{i,j,k},\;\;\gamma_{i,j,k}\delta_{i,j,k},
$$
and for all $(i,j,k)\neq (p,q,r)$:
$$
\alpha_{i,j,k}\delta_{i,j,k}-\alpha_{p,q,r}\delta_{p,q,r},\;\; \gamma_{i,j,k}\beta_{i,j,k}-\gamma_{p,q,r}\beta_{p,q,r},\;\; \beta_{i,j,k}\alpha_{p,q,r}-\delta_{i,j,k}\gamma_{p,q,r}.
$$
\end{quiv}

\noindent A key object used to determine which nontrivial extensions are possible is Cayley's hyperdeterminant $h\in S$ (see Section \ref{RepTheory}). This polynomial has weight $(-2,-2)^3$, and is the defining equation of the orbit closure $\overline{O_5}$. The two connected components of the quiver $\mathcal{Q}$ correspond to the composition factors of the modules $S_h$ and $S_h\cdot \sqrt{h}$ respectively, where $S_h$ denotes the localization of $S$ at the hyperdeterminant. We begin Section \ref{Witness} by examining the $\D$-module filtrations of these two modules.

Our motivation for understanding the category $\tn{mod}_{\tn{GL}}(\D)$ is the study of local cohomology with support in orbit closures \cite{lHorincz2018iterated, lHorincz2017equivariant, lp2018equivariant, raicu2016characters, raicu2017characters, raicu2014locals, raicu2016local, raicu2014local}. In general, for a closed subvariety $Z\subseteq V$ and a holonomic $\D$-module $M$, the local cohomology modules $\mathcal{H}^j_Z(V,M)$ are holonomic $\D$-modules (see, for instance, \cite{lyubeznik1993finiteness}). When $Z=\overline{O}$ is an orbit closure, local cohomology is a functor on $\textnormal{mod}_{\tn{GL}}(\D)$, and we conclude the paper by computing local cohomology of each simple object, with support in each orbit closure. When $O=O_{i,j,k}$ for $(i,j,k)=(1,2,2)$, $(2,1,2)$, or $(2,2,1)$,  the closure $\overline{O}$ is defined by the $2\times 2$ minors of a flattening. For example, the defining ideal of $\overline{O_{1,2,2}}$ is generated by the $2\times 2$ minors of the $2\times 4$ matrix of indeterminates $(y_{i,j})$ in $S=\tn{Sym}(A^{\ast}\otimes (B^{\ast}\otimes C^{\ast}))\cong \mathbb{C}[y_{i,j}\mid 1\leq i\leq 2, 1\leq j\leq 4]$. Thus, the local cohomology $\mathcal{H}^{\bullet}_{\overline{O_{i,j,k}}}(V,\mathcal{O}_V)$ is already known by \cite{raicu2014locals}. Using long exact sequences of cohomology and spectral sequences, these previous computations will be crucial in understanding local cohomology of simples such as $G_6$. As a consequence of our computations, we calculate all iterations of local cohomology of any simple module with support in orbit closures.\\

\noindent \textbf{Organization.} In Section \ref{Prelim} we review the necessary background on representation theory, $\D$-modules, and local cohomology. In Section \ref{Category} we prove the main theorems. We conclude the paper in Section \ref{Local} with some local cohomology computations.

\section{Preliminaries}\label{Prelim}
\subsection{Representation Theory}\label{RepTheory} Let $W$ be a complex vector space of dimension two, and write $\text{GL}(W)$ for its group of automorphisms. The irreducible finite-dimensional representations of $\text{GL}(W)$ are indexed by dominant weights $\lambda=(\lambda_1\geq \lambda_2)\in \mathbb{Z}^2$, and we write $\bS_{\lambda}W$ for the corresponding representation. We will write $\mathbb{Z}^2_{\text{dom}}$ for the set of dominant weights throughout. Given $\lambda\in \mathbb{Z}^2_{\textnormal{dom}}$, denote by $|\lambda|=\lambda_1+\lambda_2$ the sum of the entries. Note that $\bS_{\lambda}W^{\ast}=\bS_{\lambda^{\ast}}W$, where $\lambda^{\ast}=(-\lambda_2,-\lambda_1)$. Given a representation $U$ of dimension $n$, we write $\textnormal{det}(U)$ to denote the highest exterior power $\wedge^n U$. If $W$ is the vector representation, then $\textnormal{det}(W)^{\otimes r}$ is the irreducible representation $\bS_{(r,r)}W$.

Let $A$, $B$, and $C$ be two-dimensional complex vector spaces and let $\tn{GL}=\text{GL}(A)\times \text{GL}(B)\times \text{GL}(C)$. Write $\Lambda=\{\bS_{\lambda}A\otimes \bS_{\mu}B\otimes \bS_{\nu}C\mid \lambda, \mu,\nu\in \mathbb{Z}^2_{\text{dom}}\}$ for the set of isomorphism classes of finite-dimensional irreducible representations of $\tn{GL}$. We define the Grothendieck group of admissible representations $\Gamma(\tn{GL})$ to be the set of isomorphism classes of representations of the form
\begin{equation}
M=\bigoplus_{(\lambda, \mu,\nu)\in \Lambda} \left(\bS_{\lambda}A\otimes \bS_{\mu}B\otimes \bS_{\nu}C\right)^{\oplus a_{\lambda, \mu,\nu}},
\end{equation}
where $a_{\lambda, \mu,\nu}\in \mathbb{Z}_{\geq 0}$ is the multiplicity of $\bS_{\lambda}A\otimes \bS_{\mu}B\otimes \bS_{\nu}C$ in $M$. For $M$ as above, the correspondng element $[M]\in \Gamma(\tn{GL})$ is written
\begin{equation}
[M]=\sum_{(\lambda, \mu,\nu)\in \Lambda} a_{\lambda, \mu,\nu}\cdot \left[\bS_{\lambda}A\otimes \bS_{\mu}B\otimes \bS_{\nu}C\right].
\end{equation}
We write $\langle [M],[\bS_{\lambda}A\otimes \bS_{\mu}B\otimes \bS_{\nu}C]\rangle=a_{\lambda,\mu,\nu}$ for the multiplicity. A sequence $([M_r])_r$ of elements of $\Gamma(\tn{GL})$ is convergent if for every $(\lambda,\mu,\nu)\in \Lambda$, the sequence of integers $\langle [M_r],[\bS_{\lambda}A\otimes \bS_{\mu}B\otimes \bS_{\nu}C]\rangle$ is constant for $r\gg 0$. Suppose that $a_{\lambda,\mu,\nu}=\lim_{r\to \infty}\langle [M_r],[\bS_{\lambda}A\otimes \bS_{\mu}B\otimes \bS_{\nu}C]\rangle$ for all $(\lambda,\mu,\nu)\in \Lambda$. Then we write
\begin{equation}\label{limitForm}
\lim_{r\to\infty}[M_r]=\sum_{(\lambda,\mu,\nu)\in \Lambda} a_{\lambda,\mu,\nu}\cdot [\bS_{\lambda}A\otimes \bS_{\mu}B\otimes \bS_{\nu}C].
\end{equation}

In Section \ref{PrelimWit} and the proof of Lemma \ref{witD1}, we will sometimes write $[\bS_{\lambda}A\otimes \bS_{\mu}B\otimes \bS_{\nu}C]$ in the case when one or more of $\lambda,\mu,\nu\in \mathbb{Z}^2$ are not dominant. We explain what is meant by this notation here. Let $\pi\in \mathbb{Z}^2$ be one of $\lambda, \mu,\nu$. Let $\rho=(1,0)$ and consider $\pi+\rho=(\pi_1+1, \pi_2)$. Write $\textnormal{sort}(\pi+\rho)$ for the sequence of integers obtained by arranging the entries of $\pi+\rho$ in non-increasing order, and let $\tilde{\pi}=\textnormal{sort}(\pi+\rho)-\rho$. In other words, if $\pi_1+1\geq \pi_2$, then $\tilde{\pi}=\pi$, and if $\pi_2>\pi_1+1$, then $\tilde{\pi}=(\pi_2-1,\pi_1+1)$. Note that if $\pi_1+1=\pi_2$, then $\tilde{\pi}$ is not dominant, and it is dominant otherwise. Using this notation, $[\bS_{\lambda}A\otimes \bS_{\mu}B\otimes \bS_{\nu}C]$ is defined to be
\begin{equation}\label{bott}
[\bS_{\lambda}A\otimes \bS_{\mu}B\otimes \bS_{\nu}C]=
\begin{cases}
\textnormal{sgn}(\lambda)\cdot \textnormal{sgn}(\mu)\cdot \textnormal{sgn}(\nu)\cdot [\bS_{\tilde{\lambda}}A\otimes \bS_{\tilde{\mu}}B\otimes \bS_{\tilde{\nu}}C] & \textnormal{ if $\tilde{\lambda}, \tilde{\mu},\tilde{\nu}$ are dominant}\\
0 & \textnormal{otherwise}
\end{cases}
\end{equation}
where $\textnormal{sgn}(\pi)$ is the sign of the unique permutation that sorts $\pi+\rho$. 

We now recall some results about the $\tn{GL}$-equivariant structure of the polynomial ring $S$. Let $\Lambda_+$ be the set of elements $(\lambda,\mu,\nu)\in \Lambda$ with $\lambda_2,\mu_2,\nu_2\geq 0$, and recall $V=A\otimes B\otimes C$ and $S=\textnormal{Sym}(V^{\ast})$. By \cite[Proposition 4.1]{landsberg2004ideals} we have
\begin{equation}\label{charS}
[S]=\sum_{\substack{(\lambda,\mu,\nu)\in \Lambda_{+},\; d\geq0\\|\lambda|=|\mu|=|\nu|=d}}\dim_{\C}([\lambda]\otimes [\mu]\otimes [\nu])^{\Sigma_d}\cdot [\bS_{\lambda}A^{\ast}\otimes \bS_{\mu}B^{\ast}\otimes \bS_{\nu}C^{\ast}],
\end{equation}
where $[\pi]$ denotes the irreducible representation the symmetric group $\Sigma_d$ corresponding to $\pi$, and $([\lambda]\otimes [\mu]\otimes [\nu])^{\Sigma_d}$ denotes the space of $\Sigma_d$-invariants (instances of the
trivial representation) in the tensor product. We recall \cite[Corollary 4.3a]{raicu2012secant}, which allows us to compute these dimensions:
\begin{lemma}\label{raicuFormula}
Suppose that $\lambda, \mu, \nu\in \Lambda_{+}$, and $|\lambda|=|\mu|=|\nu|=d$ for some integer $d\geq 0$. Set
$$
m_{\lambda,\mu,\nu}=\dim_{\C}([\lambda]\otimes [\mu]\otimes [\nu])^{\Sigma_d},
\;\;\;f_{\lambda,\mu,\nu}=\max \{\lambda_2,\mu_2,\nu_2\},\;\;\; e_{\lambda,\mu,\nu}=\lambda_2+\mu_2+\nu_2.
$$
If $e_{\lambda,\mu,\nu}<2f_{\lambda,\mu,\nu}$, then $m_{\lambda,\mu,\nu}=0$. If $e_{\lambda,\mu,\nu}\geq d-1$, then $m_{\lambda,\mu,\nu}=\lfloor d/2\rfloor-f_{\lambda,\mu,\nu}+1$, unless $e_{\lambda,\mu,\nu}$ is odd and $d$ is even, in which case $m_{\lambda,\mu,\nu}=\lfloor d/2\rfloor-f_{\lambda,\mu,\nu}$. If $e_{\lambda,\mu,\nu}<d-1$ and $e_{\lambda,\mu,\nu}\geq 2f_{\lambda,\mu,\nu}$, then $m_{\lambda,\mu,\nu}=\lfloor (e_{\lambda,\mu,\nu}+1)/2\rfloor-f_{\lambda,\mu,\nu}+1$, unless $e_{\lambda,\mu,\nu}$ is odd, in which case $m_{\lambda,\mu,\nu}=\lfloor(e_{\lambda,\mu,\nu}+1)/2\rfloor-f_{\lambda,\mu,\nu}$.
\end{lemma}

\begin{convention}\label{convention}
In what follows, for $[M]\in \Gamma(\tn{GL})$ with $\langle [M], [\bS_{\lambda}A\otimes \bS_{\mu}B\otimes \bS_{\nu}C]\rangle \neq 0$, we will sometimes write $\lambda\times \mu\times \nu \in [M]$ or say that ``$\lambda\times \mu\times \nu$ is a weight of $M$". Under this convention, the ring $S$ consists of negative weights. For example $(0,0)^3=(0,0)\times (0,0)\times (0,0)$ and $(0,-2)\times (-1,-1)^2$ are weights of $S$.
\end{convention}

Let $h\in S$ be \defi{Cayley's hyperdeterminant}:
\begin{align*}
h & = x_{1,1,1}^2x_{2,2,2}^2 + x_{1,1,2}^2x_{2,2,1}^2 + x_{1,2,1}^2x_{2,1,2}^2 + x_{2,1,1}^2x_{1,2,2}^2 -2x_{1,1,1}x_{1,1,2}x_{2,2,1}x_{2,2,2}- 2x_{1,1,1}x_{1,2,1}x_{2,1,2}x_{2,2,2}\\
& -2x_{1,1,1}x_{1,2,2}x_{2,1,1}x_{2,2,2}-2x_{1,1,2}x_{1,2,1}x_{2,1,2}x_{2,2,1} -2x_{1,1,2}x_{1,2,2}x_{2,2,1}x_{2,1,1}-2x_{1,2,1}x_{1,2,2}x_{2,1,2}x_{2,1,1}\\
&+ 4x_{1,1,1}x_{1,2,2}x_{2,1,2}x_{2,2,1} + 4x_{1,1,2}x_{1,2,1}x_{2,1,1}x_{2,2,2}.
\end{align*}
In Section \ref{Category}, we study the $\D$-module filtrations of $S_h$ and $S_h\cdot \sqrt{h}$, where $S_h$ denotes the localization of $S$ at $h$. In order to do so, we first discuss the $\tn{GL}$-equivariant structure of these modules. Note that $h$ has weight $(-2,-2)^3$, and is the defining equation of $\overline{O_5}$. Using convention (\ref{limitForm}), we have that $[S_h]=\lim_{r\to \infty} [S\cdot h^{-r}]$. Let $(\lambda,\mu,\nu)\in \Lambda$. We have the following method to compute the multiplicity of $[\bS_{\lambda}A\otimes \bS_{\mu}B\otimes \bS_{\nu}C]$ in $[S_h]$:
\begin{equation}\label{compLocal}
\langle [S_h], [\bS_{\lambda}A\otimes \bS_{\mu}B\otimes \bS_{\nu}C]\rangle = \lim_{r\to \infty}\langle [S], [\bS_{(\lambda_1-2r,\lambda_2-2r)}A\otimes \bS_{(\mu_1-2r,\mu_2-2r)}B\otimes \bS_{(\nu_1-2r,\nu_2-2r)}C]\rangle.
\end{equation}
Also, since $\sqrt{h}$ has weight $(-1,-1)^3$, we obtain the following for $[S_h\cdot \sqrt{h}]$:
\begin{equation}\label{compSqrt}
\langle [S_h\cdot \sqrt{h}], [\bS_{\lambda}A\otimes \bS_{\mu}B\otimes \bS_{\nu}C]\rangle = \langle [S_h], [\bS_{(\lambda_1+1,\lambda_2+1)}A\otimes \bS_{(\mu_1+1,\mu_2+1)}B\otimes \bS_{(\nu_1+1,\nu_2+1)}C]\rangle.
\end{equation}
We record the following for use in Section \ref{Witness}:
\begin{lemma}\label{multSh}
Let $a\geq 0$ be an integer. If $a$ is even, then $\bS_{(a,a)}A\otimes \bS_{(a,a)}B\otimes \bS_{(a,a)}C$ has multiplicity zero in $S_h\cdot \sqrt{h}$ and multiplicity one in $S_h$. If $a$ is odd, then $\bS_{(a,a)}A\otimes \bS_{(a,a)}B\otimes \bS_{(a,a)}C$, $\bS_{(3,1)}A\otimes \bS_{(2,2)}B\otimes \bS_{(2,2)}C$, $\bS_{(2,2)}A\otimes \bS_{(3,1)}B\otimes \bS_{(2,2)}C$, and $\bS_{(2,2)}A\otimes \bS_{(2,2)}B\otimes \bS_{(3,1)}C$ have multiplicity zero in $S_h$ and multiplicity one in $S_h\cdot \sqrt{h}$.
\end{lemma}

\begin{proof} By (\ref{compLocal}), the multiplicity of $\bS_{(a,a)}A\otimes \bS_{(a,a)}B\otimes \bS_{(a,a)}C$ in $S_h$ is given by 
\begin{equation}\label{tempLimit}
\lim_{r\to \infty}\langle [S], [\bS_{(a-2r,a-2r)}A\otimes \bS_{(a-2r,a-2r)}B\otimes \bS_{(a-2r,a-2r)}C]\rangle.
\end{equation}
Since $\bS_{(a-2r,a-2r)}A=\bS_{(2r-a,2r-a)}A^{\ast}$, we compute the multiplicity of $\bS_{(2r-a,2r-a)}A^{\ast}\otimes \bS_{(2r-a,2r-a)}B^{\ast}\otimes \bS_{(2r-a,2r-a)}C^{\ast}$ in $S$ for $r\gg 0$. Let $(\lambda,\mu,\nu)=((2r-a,2r-a),(2r-a,2r-a),(2r-a,2r-a))$, and use notation from Lemma \ref{raicuFormula}, setting $d=4r-2a$. In this situation, $f_{\lambda,\mu,\nu}=2r-a$, and $e_{\lambda,\mu,\nu}=6r-3a$. For $r\gg 0$, we have $e_{\lambda,\mu,\nu}\geq d-1$ and $e_{\lambda,\mu,\nu}\geq 2f_{\lambda,\mu,\nu}$. Thus by Lemma \ref{raicuFormula}, it follows that (for $r\gg 0$) $m_{\lambda,\mu,\nu}=1$ if $a$ is even and $m_{\lambda,\mu,\nu}=0$ if $a$ is odd. By (\ref{charS}) and (\ref{tempLimit}), it follows that if $a$ is even, then $\bS_{(a,a)}A\otimes \bS_{(a,a)}B\otimes \bS_{(a,a)}C$ appears in $S_h$ with multiplicity one, and if $a$ is odd, then $\bS_{(a,a)}A\otimes \bS_{(a,a)}B\otimes \bS_{(a,a)}C$ does not appear in $S_h$. 

The assertions about the multiplicity of $\bS_{(a,a)}A\otimes \bS_{(a,a)}B\otimes \bS_{(a,a)}C$ in $S_h\cdot \sqrt{h}$ follow from the first paragraph and (\ref{compSqrt}). We now prove the assertion about $\bS_{(3,1)}A\otimes \bS_{(2,2)}B\otimes \bS_{(2,2)}C$, leaving the similar results about $\bS_{(2,2)}A\otimes \bS_{(3,1)}B\otimes \bS_{(2,2)}C$, and $\bS_{(2,2)}A\otimes \bS_{(2,2)}B\otimes \bS_{(3,1)}C$ to the reader. By (\ref{compLocal}), to compute the multiplicity of $\bS_{(3,1)}A\otimes \bS_{(2,2)}B\otimes \bS_{(2,2)}C$ in $S_h$, we need to determine
\begin{equation}\label{tempLimit1}
\lim_{r\to \infty}\langle [S], [\bS_{(3-2r,1-2r)}A\otimes \bS_{(2-2r,2-2r)}B\otimes \bS_{(2-2r,2-2r)}C]\rangle.
\end{equation}
Due to the relation $\bS_{(\lambda_1,\lambda_2)}A=\bS_{(-\lambda_2,-\lambda_1)}A^{\ast}$, we compute the multiplicity of $\bS_{(2r-1,2r-3)}A^{\ast}\otimes \bS_{(2r-2,2r-2)}B^{\ast}\otimes \bS_{(2r-2,2r-2)}C^{\ast}$ in $S$ for $r\gg 0$. Set $(\lambda,\mu,\nu)=((2r-1,2r-3),(2r-2,2r-2),(2r-2,2r-2))$. Again, using the notation from Lemma \ref{raicuFormula}, we have $d=4r-4$, $f_{\lambda,\mu,\nu}=2r-2$, and $e_{\lambda,\mu,\nu}=6r-7$. For $r\gg 0$, $e_{\lambda,\mu,\nu}\geq 2f_{\lambda,\mu,\nu}$ and $e_{\lambda,\mu,\nu}\geq d-1$. Thus (for $r\gg 0$) $m_{\lambda,\mu,\nu}=0$, so by (\ref{charS}) and (\ref{tempLimit1}), the multiplicity of $\bS_{(3,1)}A\otimes \bS_{(2,2)}B\otimes \bS_{(2,2)}C$ in $S_h$ is zero, as claimed. Finally, we show that the multiplicity of $\bS_{(3,1)}A\otimes \bS_{(2,2)}B\otimes \bS_{(2,2)}C$ in $S_h\cdot \sqrt{h}$ is one. Using (\ref{compLocal}) and (\ref{compSqrt}), we need to show that the following limit is equal to one:
\begin{equation}\label{tempLimit2}
\lim_{r\to \infty}\langle [S], [\bS_{(2r-2,2r-4)}A^{\ast}\otimes \bS_{(2r-3,2r-3)}B^{\ast}\otimes \bS_{(2r-3,2r-3)}C^{\ast}]\rangle
\end{equation}
Notice that we have written this expression in terms of dual representations, so that we may apply Lemma \ref{raicuFormula}. Let $(\lambda,\mu,\nu)=((2r-2,2r-4),(2r-3,2r-3),(2r-3,2r-3))$, so that $d=4r-6$, $f_{\lambda,\mu,\nu}=2r-3$, and $e_{\lambda,\mu,\nu}=6r-10$. By Lemma \ref{raicuFormula}, we obtain that $m_{\lambda,\mu,\nu}=1$ for $r\gg 0$. Thus, the limit (\ref{tempLimit2}) is equal to one, as needed.
\end{proof}

\subsection{Equivariant $\mathcal{D}$-modules and Local Cohomology}\label{D}
In this section, let $V$ be a finite dimensional complex vector space, thought as an affine space, and let $G$ be a connected linear algebraic group acting on $V$. Let $\D$=$\D_V$ be the Weyl algebra of differential operators with polynomial coefficients. Consider the category $\textnormal{mod}(\D)$ of finitely-generated left $\D$-modules with $\D$-linear maps as morphisms. A $\D$-module $M$ is $G$-equivariant if there is a $\D_{G\times V}$-module isomorphism $\tau:p^{\ast}M\to m^{\ast}M$, where $p:G\times V\to V$ is the projection, $m:G\times V\to V$ is the multiplication map, and $\tau$ satisfies the co-cycle conditions (see \cite[Definition 11.5.2]{hotta2007d}). Write $\textnormal{mod}_G(\D)$ for the full subcategory of finitely-generated $G$-equivariant (left) $\D$-modules. By \cite[Proposition 3.1.2]{van1999local}, morphisms in $\textnormal{mod}_G(\D)$ are automatically $G$-equivariant.

For $Z\subseteq V$ a $G$-stable closed subvariety, write $\textnormal{mod}^Z_G(\D)$ for the full subcategory of $\textnormal{mod}_G(\D)$ consisting of modules with support contained in $Z$. By the Riemann-Hilbert correspondence, simple objects in $\textnormal{mod}^Z_G(\D)$ correspond to $G$-equivariant irreducible local systems on open subsets of $Z$. Each one yields a simple object in $\textnormal{mod}_G(\D)$. Let $\mathcal{M}$ be a $G$-equivariant irreducible local system on an open subset of $Z$ contained in the smooth locus, and write $\mathcal{L}(Z,\mathcal{M},V)$ for the simple object corresponding to $\mathcal{M}$. This is called the intersection homology $\D$-module. When $\mathcal{M}$ is the trivial local system, we simply write $\mathcal{L}(Z,V)$ for the intersection homology $\D$-module of $\mathcal{M}$. By \cite[Theorem 11.6.1]{hotta2007d}, we have the following:
\begin{theorem}\label{eqRH}
Suppose that $G$ acts on $V$ with finitely many orbits. 

(a) There is a bijective correspondence
$$
\left\{(O,\mathcal{M})\mid \substack{\text{$O$ is a $G$-orbit,}\\ \text{$\mathcal{M}$ is an equivariant irreducible local system on $O$}}\right\}\leftrightarrow \{\text{simple objects in $\textnormal{mod}_G(\D)$}\}
$$
where $(O,\mathcal{M})$ corresponds to $\mathcal{L}(O,\mathcal{M},V)$. 

(b) Moreover, if we fix an orbit $O=G/K$ and set $K^0$ to be the connected component of the identity in $K$, then there is a bijective correspondence
$$
\{\text{equivariant irreducible local systems on $O$}\}\leftrightarrow \{\text{irreducible representations of $K/K^0$}\}.
$$
\end{theorem}

The group $K/K^0$ is called the \defi{component group} of $O$. By this theorem, in order to determine the number of simple $\tn{GL}$-equivariant $\D_V$-modules supported on each orbit closure, we just need to compute the component group corresponding to each orbit. We begin Section \ref{Category} with this analysis.

Let $Z\subseteq V$ be closed and let $U=V\setminus Z$, with open immersion $j:U\hookrightarrow V$. Write $j_{\ast}$ and $j^{\ast}$ to denote the direct and inverse image functors of quasi-coherent sheaves. These functors restrict to functors between $\textnormal{mod}(\D_U)$ and $\textnormal{mod}(\D_V)$, and $j_{\ast}$ is right adjoint to $j^{\ast}$. The adjunction gives a map $M\to j_{\ast}j^{\ast}M$, yielding an exact sequence and isomorphisms
\begin{equation}\label{locSES}
0\longrightarrow \mathcal{H}^0_Z(V,M)\longrightarrow M\longrightarrow j_{\ast}j^{\ast}M\longrightarrow \mathcal{H}^1_Z(V,M)\longrightarrow 0,\;\;\;\;\;R^kj_{\ast}(j^{\ast}M)\cong \mathcal{H}^{k+1}_Z(V,M),
\end{equation}
for $k\geq 1$, where $\mathcal{H}_Z^i(V,M)$ denotes the $i$-th local cohomology of $M$ with support in $Z$, and $R^kj_{\ast}$ denotes the $k$-th derived functor of $j_{\ast}$. Note that for $M=\mathcal{O}_V$ we have
\begin{equation}\label{codim}
\mathcal{H}^j_Z(V,\mathcal{O}_V)=0\;\;\tn{ for all $j<\textnormal{codim}(Z,V)$,}\;\;\;\tn{and}\;\;\;\mathcal{H}^{\textnormal{codim}(Z,V)}_Z(V,\mathcal{O}_V)\neq 0.
\end{equation}
We recall the following general fact (see, for example \cite[Lemma 3.11]{lHorincz2018categories} or \cite[Page 9]{raicu2017characters}):

\begin{prop}\label{IHmod}
Using the notation above, set $c=\textnormal{codim}(Z,V)$. The intersection homology $\D$-module $\mathcal{L}(Z,V)$ is the unique simple submodule of $\mathcal{H}^c_Z(V,\mathcal{O}_V)$, and all composition factors of $\mathcal{H}^c_Z(V,\mathcal{O}_V)/\mathcal{L}(Z,V)$ have support contained in the singular locus of $Z$. Further, all composition factors of $\mathcal{H}^j_Z(V,\mathcal{O}_V)$ for $j>c$ have support contained in the singular locus of $Z$.
\end{prop}

\noindent This will be used in Section \ref{Local} to compute local cohomology of simple objects with support in orbit closures, in conjunction with homological techniques such as spectral sequences and long exact sequences. Since we are only working with cohomology on a single (affine) space $V$, we will write $H^j_Z(M)=\mathcal{H}^j_Z(V,M)$ throughout. 

We now discuss two functors that will be crucial in our study of $\textnormal{mod}_{\tn{GL}}(\D)$. Let $V=A\otimes B\otimes C$ and $\tn{GL}$ be as above. There is a self-equivalence of categories $\mathcal{F}$ on $\textnormal{mod}_{\tn{GL}}(\D_V)$, given by $\mathcal{F}(M)=M^{\ast}\otimes_{\C} \textnormal{det}(V)$ (see \cite[Section 4.3]{lHorincz2018categories}). We will refer to this functor as the Fourier Transform. In our situation, $\textnormal{det}(V)=\bS_{(4,4)}A\otimes \bS_{(4,4)}B\otimes \bS_{(4,4)}C$. By abuse of notation, consider the function $\mathcal{F}:\mathbb{Z}^2 \to \mathbb{Z}^2$ given by 
$$
\mathcal{F}(\lambda)=\lambda^{\ast}+(4,4)=(-\lambda_2+4,-\lambda_1+4).
$$
This induces a function $\mathcal{F}:\Gamma(\tn{GL})\to \Gamma(\tn{GL})$ given by
\begin{equation}\label{fourier}
\mathcal{F}\left(\sum_{(\lambda, \mu,\nu)\in \Lambda} a_{\lambda, \mu,\nu}\cdot \left[\bS_{\lambda}A\otimes \bS_{\mu}B\otimes \bS_{\nu}C\right]\right)
=\sum_{(\lambda, \mu,\nu)\in \Lambda} a_{\lambda, \mu,\nu}\cdot \left[\bS_{\mathcal{F}(\lambda)}A\otimes \bS_{\mathcal{F}(\mu)}B\otimes \bS_{\mathcal{F}(\nu)}C\right].
\end{equation}
Notice that for any object $M$ of $\textnormal{mod}_{\tn{GL}}(\D_V)$, we have $[\mathcal{F}(M)]=\mathcal{F}([M])$. In general, if $M$ is simple, then $\mathcal{F}(M)$ will also be simple. For example, the Fourier transform of $S$ is $E=\textnormal{Sym}(V)\otimes \textnormal{det}(V)$. This allows us to compute multiplicities $\langle [E], [\bS_{\lambda}A\otimes \bS_{\mu}B\otimes \bS_{\nu}C]\rangle$. We immediately conclude the following:
\begin{lemma}\label{notE}
The representation $\bS_{(4,4)}A\otimes \bS_{(4,4)}B\otimes \bS_{(4,4)}C$ has multiplicity one in $E$, and for $a\leq 3$ the representation $\bS_{(a,a)}A\otimes \bS_{(a,a)}B\otimes \bS_{(a,a)}C$ does not appear in $E$. Finally, the representations $\bS_{(3,1)}A\otimes \bS_{(2,2)}B\otimes \bS_{(2,2)}C$, $\bS_{(2,2)}A\otimes \bS_{(3,1)}B\otimes \bS_{(2,2)}C$, and $\bS_{(2,2)}A\otimes \bS_{(2,2)}B\otimes \bS_{(3,1)}C$ do not appear in $E$.
\end{lemma}

There is another functor on $\textnormal{mod}_{\tn{GL}}(\D)$ that permutes simple objects, the holonomic duality functor $\mathbb{D}$ (see \cite[Section 2.6]{hotta2007d}). The duality functor is an equivalence of categories between $\textnormal{mod}_{\tn{GL}}(\D)$ and $\textnormal{mod}_{\tn{GL}}(\D)^{\textnormal{op}}$, defined via
$$
\mathbb{D}(M)=\mathscr{E}\E\textnormal{ }^8_{\D_V}(M,\D_V)\otimes_{\mathcal{O}_V} \omega_V^{-1},
$$
where $8=\dim V$ and $\omega_V$ is the canonical bundle. Via the Riemann-Hilbert correspondence, the holonomic duality functor is the Verdier duality functor on perverse sheaves. Thus, for a simple $\D$-module $M$ corresponding to an irreducible local system $\mathcal{M}$ on an orbit $O$, the duality functor sends $M$ to the $\D$-module corresponding to the dual local system $\mathcal{M}^{\ast}$.

\subsection{Representations of quivers}\label{PrelimQuiver}

A quiver $\mathcal{Q}=(\mathcal{Q}_0,\mathcal{Q}_1)$ is an oriented graph with a finite set of vertices $\mathcal{Q}_0$ and a finite set of arrows $\mathcal{Q}_1$. An arrow $\alpha\in \mathcal{Q}_1$ has source $s(\alpha)\in \mathcal{Q}_0$ and a target $t(\alpha)\in \mathcal{Q}_0$. A directed path $p$ in $\mathcal{Q}$ from $a$ to $b$ is a sequence of arrows $\alpha_1,\cdots, \alpha_k$ such that $s(\alpha_1)=a$, $t(\alpha_k)=b$, and $s(\alpha_i)=t(\alpha_{i-1})$. A relation in $\mathcal{Q}$ is a linear combination of paths of length at least two having the same source and target. A quiver with relations $(\mathcal{Q},\mathcal{I})$ is a quiver $\mathcal{Q}$ together with a finite set of relations $\mathcal{I}$. A finite-dimensional representation $W$ of a quiver with relations $(\mathcal{Q},\mathcal{I})$ is a collection of finite-dimensional vector spaces $\{W_a\mid a\in \mathcal{Q}_0\}$ indexed by $\mathcal{Q}_0$, along with a set of linear maps $\{W(\alpha):W_{s(\alpha)}\to W_{t(\alpha)}\mid \alpha\in \mathcal{Q}_1\}$ satisfying the relations in $\mathcal{I}$.

The category $\textnormal{mod}_{\tn{GL}}(\D)$ is equivalent to the category of finite-dimensional representations of a quiver with relations $(\mathcal{Q},\mathcal{I})$ (for the most direct proof in our situation, see \cite[Proposition 2.5]{lHorincz2018categories}). We now gather a few facts for later, to be used when we determine the quiver. Given $M\in \textnormal{mod}_{\tn{GL}}(\D)$, write $\mathcal{W}^M$ for the corresponding representation of $(\mathcal{Q},\mathcal{I})$.
\begin{enumerate}
\item There is a one-to-one correspondence between simple objects in $\textnormal{mod}_{\tn{GL}}(\D)$ and vertices of $\mathcal{Q}$.

\item If $M\in \textnormal{mod}_{\tn{GL}}(\D)$ and $N$ is a simple composition factor of $M$ with multiplicity $d$, then the representation $\mathcal{W}^M\in \textnormal{rep}(\mathcal{Q},\mathcal{I})$ has  $\dim \mathcal{W}^M_n=d$, where $n$ is the vertex corresponding to $N$. 

\item If $M$ and $N$ are simple objects in $\textnormal{mod}_{\tn{GL}}(\D)$, corresponding to vertices $m$ and $n$ in $\mathcal{Q}$, then the number of arrows from $m$ to $n$ is equal to $\dim_{\C}\textnormal{Ext}^1_{\mathcal{D}}(M,N)$.

\item Recall the Fourier transform $\mathcal{F}$ and the holonomic duality functor $\mathbb{D}$ from Section 2.2. Given a vertex $m$ of $\mathcal{Q}$ corresponding to a simple module $M$, write $\mathcal{F}(m)$ (resp, $\mathbb{D}(m)$) for the vertex of $\mathcal{Q}$ corresponding to $\mathcal{F}(M)$ (resp. $\mathbb{D}(M)$). If $n$ is a vertex of $\mathcal{Q}$ corresponding to a simple module $N$ in $\textnormal{mod}_{\tn{GL}}(\D)$, then the number of arrows from $m$ to $n$ is equal to the number of arrows from $\mathbb{D}(n)$ to $\mathbb{D}(m)$ and the number of arrows from $\mathcal{F}(m)$ to $\mathcal{F}(n)$.

\item Let $I^M$ be the injective hull in $\tn{mod}_{\tn{GL}}(\D)$ of a simple module $M$, and let $N$ be another simple module. The number of paths from $n$ to $m$ is equal to the multiplicity of $N$ as a composition factor of $I^M$. Dually, if $P^M$ is the projective cover of $M$, the number of paths from $m$ to $n$ is equal to the multiplicity of $N$ as a composition factor of $P^M$.
\end{enumerate}

\subsection{Witness weights}\label{PrelimWit}

In order to determine the quiver structure of the category $\textnormal{mod}_{\tn{GL}}(\D_V)$, an important tool will be to know a weight unique to each simple equivariant $\D$-module $M$. In other words, for such $M$ we want to find $(\lambda,\mu,\nu)\in \Lambda$ such that $\langle [M], [\bS_{\lambda}A\otimes \bS_{\mu}B\otimes \bS_{\nu}C]\rangle \neq 0$ and $\langle [N], [\bS_{\lambda}A\otimes \bS_{\mu}B\otimes \bS_{\nu}C]\rangle =0$ for all simple equivariant $\D$-modules $N\neq M$. We will call these \defi{witness weights}. Recall the simple modules $D_0=E$, $D_1$, $D_{1,2,2}$, $D_{2,1,2}$, $D_{2,2,1}$, $D_5$, $S$, and $G_6$ from the Classification of Simple Modules in the introduction. In this section, we discuss how to obtain (any) weights in $D_1$ and $D_{i,j,k}$ for $(i,j,k)=(1,2,2)$, $(2,1,2)$, and $(2,2,1)$. We determine the witness weights in Section \ref{Witness}.

To obtain weights of $D_1=\mathcal{L}(O_1,V)$, we will push forward the structure sheaf of $O_1$ from a desingularization of $\overline{O_1}$. The desingularization will be a homogeneous vector bundle on a product of projective spaces, and the problem of computing the desired weights will reduce to computing Euler characteristics of vector bundles on this product (as admissible representations). This is the technique used by Raicu to compute characters of simple equivariant $\D$-modules on Veronese cones and spaces of matrices (generic, symmetric, skew-symmetric) \cite{raicu2016characters, raicu2017characters}.

The setup is as follows: Let $X=\mathbb{P}(A)\times \mathbb{P}(B)\times \mathbb{P}(C)$ (we write $\mathbb{P}(A)$ for the variety of one-dimensional subspaces in $A$), with projections $p_W:X\to \mathbb{P}(W)$ for $W=A$, $B$, $C$, and let $Y=\text{Tot}_X(p_A^{\ast}\mathcal{O}(1)\otimes p_B^{\ast}\mathcal{O}(1)\otimes p_C^{\ast}\mathcal{O}(1))$. Consider the following diagram:
$$
\begin{tikzcd}
Y \arrow[r, hook, "s"] \arrow[dr, "\pi"] & V\times X\arrow[d, "p"]\\
& V
\end{tikzcd}
$$
where $s$ is the inclusion, $p$ is the projection, and $\pi=p\circ s$ is the composition. Then $Y$ is a desingularization of $\overline{O_1}$ with $\pi^{-1}(O_1)\cong O_1$. Indeed, $\pi^{-1}(O_0)\cong X$ and for all nonzero pure tensors $v=a\otimes b\otimes c\in O_1$, we have $\pi^{-1}(v)=(v,(\langle a \rangle,\langle b\rangle, \langle c \rangle))\in Y$. Using notation from \cite[Chapter 1.5]{hotta2007d}, we write $\int_{\pi}\mathcal{O}_{O_1}$ for the $\mathcal{D}$-module direct image from $Y$ of the structure sheaf $\mathcal{O}_{O_1}$, and we write $\int_{\pi}^j\mathcal{O}_{O_1}$ for its $j$-th cohomology. Since $\pi$ induces a birational isomorphism from $Y$ to $\overline{O_1}$ away from the origin, it follows that the cohomology of $\int_{\pi}\mathcal{O}_{O_1}$ will have support contained in $\overline{O_1}$, and $D_1$ will appear. Thus, we may obtain information about the weights that appear in $D_1$ from knowledge of the weights that appear in $D_0=E$ and the $\tn{GL}$-admissible Euler characteristic:
\begin{equation}\label{admissibleEuler}
\left[\chi\left(\int_{\pi}\mathcal{O}_{O_1}\right)\right]=\sum_{j\in \mathbb{Z}} (-1)^j\left[ \int_{\pi}^k\mathcal{O}_{O_1}\right].
\end{equation}

By \cite[Proposition 2.10]{raicu2016characters},we have the following in $\Gamma(\tn{GL})$:
\begin{equation}\label{eulerChar}
\left[ \chi\left(\int_{\pi}\mathcal{O}_{O_1}\right)\right]=\lim_{r\to \infty}\left(\sum_{i=0}^{3}(-1)^{3-i}\cdot [\chi(X,\Omega^i_X\otimes \mathcal{L}^r)\otimes  E]\right),
\end{equation}
where $\mathcal{L}=p_A^{\ast}\mathcal{O}(-1)\otimes p_B^{\ast}\mathcal{O}(-1)\otimes p_C^{\ast}\mathcal{O}(-1)$. For the remainder of the section, recall the convention (\ref{bott}). Let $[2]=\{1,2\}$ and let $\binom{[2]}{1}$ denote the set of subsets of $[2]$ of size one. In other words, $\binom{[2]}{1}=\{\{1\},\{2\}\}$. For $r\in \mathbb{Z}$ and $I\in \binom{[2]}{1}$, write $(r^I)\in \mathbb{Z}^2$ for the tuple with $r$ in the $I$-th place and zero elsewhere. If $I=\{1\}$, then $(r^I)=(r,0)$ and if $I=\{2\}$, then $(r^I)=(0,r)$. By \cite[Lemma 2.5]{raicu2016characters}, we have that
$$
\sum_{i=0}^{3}(-1)^{3-i}\cdot [\chi(X,\Omega^i_X\otimes \mathcal{L}^r)]=-[p(V)],\;\;\textnormal{where }\;[p(V)]=\sum_{I,J,K\in \binom{[2]}{1}} [\bS_{(r^I)}A\otimes \bS_{(r^J)}B\otimes \bS_{(r^K)}C].
$$
Given $\lambda\in \mathbb{Z}^2_{\textnormal{dom}}$, write $\lambda(r,I)=\lambda+(r^I)$. Let $(\lambda, \mu,\nu)\in \Lambda$. Combining the above, we obtain
$$
\sum_{i=0}^{3}(-1)^{3-i}\cdot \left\langle[\chi(X,\Omega^i_X\otimes \mathcal{L}^r)\otimes E], [\bS_{\lambda}A\otimes \bS_{\mu}B\otimes \bS_{\nu}C]\right\rangle =- \langle [p(V)]\otimes [E], [\bS_{\lambda}A\otimes \bS_{\mu}B\otimes \bS_{\nu}C]\rangle
$$
$$
=-\sum_{I,J,K\in \binom{[2]}{1}}\langle [\textnormal{Sym}(V)], [\bS_{(\lambda_1-4,\lambda_2-4)(r,I)}A\otimes \bS_{(\mu_1-4,\mu_2-4)(r,J)}B\otimes \bS_{(\nu_1-4,\nu_2-4)(r,K)}C]\rangle,
$$
where the second equality follows from \cite[Lemma 2.3]{raicu2016characters}. We summarize with the following:

\begin{lemma}\label{eulerForm}
The multiplicity of $[\bS_{\lambda}A\otimes \bS_{\mu}B\otimes \bS_{\nu}C]$ in $[\chi(\int_{\pi}\mathcal{O}_{O_1})]$ is given by
$$
\lim_{r\to \infty}\left( -\sum_{I,J,K\in \binom{[2]}{1}}\langle [\textnormal{Sym}(V)], [\bS_{(\lambda_1-4,\lambda_2-4)(r,I)}A\otimes \bS_{(\mu_1-4,\mu_2-4)(r,J)}B\otimes \bS_{(\nu_1-4,\nu_2-4)(r,K)}C]\rangle\right),
$$
where $\pi(r,I)=\pi+(r^I)$ for $\pi\in \mathbb{Z}^2_{\tn{dom}}$.
\end{lemma}
\noindent We will use this in the proof of Lemma \ref{witD1}.

Next, we find some weights that appear in $D_{i,j,k}=\mathcal{L}(\overline{O_{i,j,k}},V)$ with multiplicity one. Let $W_1$ and $W_2$ be complex vector spaces of dimensions $n$ and $m$ respectively, and consider $W_1\otimes W_2$ the space of $n\times m$ matrices. This space has a natural action of $\textnormal{GL}(W_1)\times \textnormal{GL}(W_2)$ with $k+1$ orbits, where $k=\min(m,n)$. In this case, the component groups corresponding to each orbit are trivial, and by Theorem \ref{eqRH}, there are $k+1$ simple $\textnormal{GL}(W_1)\times \textnormal{GL}(W_2)$-equivariant $\D_{W_1\otimes W_2}$-modules. C. Raicu has computed the $\textnormal{GL}(W_1)\times \textnormal{GL}(W_2)$ structure of these simple modules \cite{raicu2016characters}. We will use these computations to obtain information about the $\textnormal{GL}(A)\times \textnormal{GL}(B)\times \textnormal{GL}(C)$ structure of our simple modules $D_{i,j,k}$ for $(i,j,k)=(1,2,2)$, $(2,1,2)$, and $(2,2,1)$. In particular, we will obtain the witness weights from these previous computations.

The space $V=A\otimes B\otimes C$ may be identified with the space of $2\times 4$ matrices $A\otimes (B\otimes C)$. Under this identification, the orbit closure $\overline{O_{1,2,2}}$ is the determinantal variety of $2\times 4$ matrices of rank $\leq 1$. The similar results hold for $\overline{O_{2,1,2}}$ and $\overline{O_{2,2,1}}$. Let $\mathcal{A}=\{\lambda\in \mathbb{Z}^2_{\textnormal{dom}}\mid \lambda_1\geq 3,\lambda_2\leq 1\}$. Given $\lambda\in \mathcal{A}$, write $\lambda(1)=(\lambda_1-2,1,1,\lambda_2)\in \mathbb{Z}^4_{\textnormal{dom}}$. By \cite[Section 3.2]{raicu2016characters}, the simple $\D_V$-module $D_{1,2,2}$ decomposes as a representation of $\textnormal{GL}(A)\times \textnormal{GL}(B\otimes C)$ as follows:
\begin{equation}
D_{1,2,2}=\bigoplus_{\lambda\in \mathcal{A}} \bS_{\lambda} A\otimes \bS_{\lambda(1)}(B\otimes C)
\end{equation}
Similar decompositions hold for $D_{2,1,2}$ and $D_{2,2,1}$. Notice that if $\lambda=(3,1)$, then $\lambda(1)=(1,1,1,1)$. Therefore $\bS_{(3,1)}A\otimes \bS_{(1,1,1,1)}(B\otimes C)=\bS_{(3,1)}A\otimes \bS_{(2,2)}B\otimes \bS_{(2,2)}C$ appears in $D_{1,2,2}$ with multiplicity one. We conclude:

\begin{lemma}\label{multDet}
The following hold in $\Gamma(\tn{GL})$:
$$
\langle [D_{1,2,2}], [\bS_{(3,1)}A\otimes \bS_{(2,2)}B\otimes \bS_{(2,2)}C]\rangle=1,\;\;\langle [D_{2,1,2}], [\bS_{(2,2)}A\otimes \bS_{(3,1)}B\otimes \bS_{(2,2)}C]\rangle=1,
$$
$$
\langle [D_{2,2,1}], [\bS_{(2,2)}A\otimes \bS_{(2,2)}B\otimes \bS_{(3,1)}C]\rangle=1.
$$
In addition, for $(i,j,k)=(1,2,2)$, $(2,1,2)$, $(2,2,1)$ and $a\in \mathbb{Z}$, the representation $[\bS_{(a,a)}A\otimes \bS_{(a,a)}B\otimes \bS_{(a,a)}C]$ does not appear in $[D_{i,j,k}]$.
\end{lemma}

\noindent The second assertion follows from the fact that, for all $a\in \mathbb{Z}$, the weight $(a,a)$ does not belong to the set $\mathcal{A}$.

\section{The category $\textnormal{mod}_{\tn{GL}}(\D_V)$ }\label{Category}
In this section we prove the main theorems. We begin by classifying the simple modules and computing the witness weights. We go on to determine the quiver structure of the category $\textnormal{mod}_{\tn{GL}}(\D)$.
\subsection{Component groups for the orbits}

By Theorem \ref{eqRH}, the simple objects in $\textnormal{mod}_{\tn{GL}}(\D)$ are in one-to-one correspondence with representations of the component groups of each orbit. We now compute the component groups, immediately yielding the first assertion of the Classification of Simple Modules. Note first that the component group corresponding to $O_1$ is trivial by \cite[Lemma 4.13]{lHorincz2018categories}, as $O_1$ is the orbit of the highest weight vector of $V$. For the following computations, we will consider an element of the group:
\begin{equation}\label{gel}
g=\left(X,Y,Z\right)=\left(\left( {\begin{array}{cc}
   x_{1,1} & x_{1,2} \\
   x_{2,1} & x_{2,2}
  \end{array} } \right),
\left( {\begin{array}{cc}
   y_{1,1} & y_{1,2} \\
   y_{2,1} & y_{2,2}
  \end{array} } \right),
 \left( {\begin{array}{cc}
   z_{1,1} & z_{1,2} \\
   z_{2,1} & z_{2,2} 
  \end{array} } \right)\right)\in \tn{GL}.
\end{equation}
Given an orbit $O$ and $v\in O$, we will determine what conditions are imposed on $x_{i,j}$, $y_{i,j}$, and $z_{i,j}$ if $g$ is in the isotropy of $v$. These equations are used to find the connected components of the isotropy.

\begin{lemma}
The $\tn{GL}$-isotropy subgroups for $a_1\otimes b_1\otimes c_1+a_1\otimes b_2\otimes c_2$, $a_1\otimes b_1\otimes c_1+a_2\otimes b_1\otimes c_2$, and $a_1\otimes b_1\otimes c_1+a_2\otimes b_2\otimes c_1$ are path connected. In particular, the component groups corresponding to $O_{1,2,2}$, $O_{2,1,2}$, and $O_{2,2,1}$ are trivial.
\end{lemma}

\begin{proof}
By symmetry, it suffices to prove the result for $O_{1,2,2}$. Let $v=a_1\otimes b_1\otimes c_1+a_1\otimes b_2\otimes c_2$ and $g\cdot v=\sum_{i,j,k}f_{i,j,k}a_i\otimes b_j\otimes c_k$, where $g$ is as in (\ref{gel}). If $g$ is in the isotropy of $v$, then $f_{1,1,1}=1$, $f_{1,2,2}=1$, and $f_{i,j,k}=0$ otherwise. Since $f_{1,1,1}=1$, we have that $y_{1,1}z_{1,1}+y_{1,2}z_{1,2}$ is nonzero. Thus, $x_{2,1}=0$, as $f_{2,1,1}=x_{2,1}(y_{1,1}z_{1,1}+y_{1,2}z_{1,2})$. We conclude that the system $f_{1,1,1}=1$, $f_{1,2,2}=1$, and $f_{i,j,k}=0$ ($(i,j,k)\neq (1,1,1)$, $(1,2,2)$) is equivalent to the vanishing of the following equations:
$$
x_{2,1},\;\;x_{1,1}y_{1,1}z_{1,1}+x_{1,1}y_{1,2}z_{1,2}-1,\;\; x_{1,1}y_{2,1}z_{2,1}+x_{1,1}y_{2,2}z_{2,2}-1,\;\;\; y_{1,1}z_{2,1}+y_{1,2}z_{2,2},\;\;\; y_{2,1}z_{1,1}+y_{2,2}z_{1,2}.
$$
Since $g\in \tn{GL}$, this forces $x_{1,1}$ to be nonzero, and the equations above imply that the matrix product $Y^{T}\cdot Z$ is $1/x_{1,1}\cdot \tn{Id}_2$, where $\tn{Id}_2$ is the $2\times 2$ identity matrix. Thus, $Z$ is determined by $x_{1,1}$ and $Y$. We conclude that the isotropy subgroup for $v$ is isomorphic to $(\mathbb{C}^{\ast})^2\times \mathbb{C}\times \tn{GL}_2(\mathbb{C})$, where the two copies of $\mathbb{C}^{\ast}$ correspond to the coordinates $x_{1,1}$ and $x_{2,2}$, the copy of $\mathbb{C}$ corresponds to $x_{1,2}$, and the copy of $\tn{GL}_2(\mathbb{C})$ corresponds to the $Y$ coordinates. Since $(\mathbb{C}^{\ast})^2\times \mathbb{C}\times \tn{GL}_2(\mathbb{C})$ is path connected, the result follows.
\end{proof}

\begin{lemma}\label{compO5}
The $\tn{GL}$-isotropy subgroup for $v=a_1\otimes b_1\otimes c_1+a_1\otimes b_2\otimes c_2+a_2\otimes b_1\otimes c_2$ is connected. In particular, the component group corresponding to $O_5$ is trivial
\end{lemma}

\begin{proof}
Similar to the previous proof, write $g\cdot v=\sum_{i,j,k}f_{i,j,k}a_i\otimes b_j\otimes c_k$. The condition that $g$ is in the isotropy of $v$ is eight equations: $f_{1,1,1}=1$, $f_{1,2,2}=1$, $f_{2,1,2}=1$, and $f_{i,j,k}=0$ otherwise. In this case, $x_{2,1}=y_{2,1}=z_{1,2}=0$, and modulo these variables the system of equations 
$$
\left\{f_{1,1,1}=1, f_{1,2,2}=1, f_{2,1,2}=1, f_{i,j,k}=0 \mid (i,j,k)\neq (1,1,1),(1,2,2),(2,1,2)\right\}
$$
is equivalent to the vanishing of the following equations:
\begin{equation}\label{O5eq}
x_{1,1}y_{1,1}z_{1,1}-1,\;\;\;\; x_{1,1}y_{2,2}z_{2,2}-1,\;\;\;\;x_{2,2}y_{1,1}z_{2,2}-1,\;\;\;\;
x_{1,1}y_{1,1}z_{2,1}+x_{1,1}y_{1,2}z_{2,2}+x_{1,2}y_{1,1}z_{2,2}.
\end{equation}
It suffices to show that the variety $T$ defined by the equations (\ref{O5eq}) in $\mathbb{C}^9$ is path connected. Choose a point $P=(X_{1,1},X_{1,2},X_{2,2},Y_{1,1},Y_{1,2},Y_{2,2},Z_{1,1},Z_{2,1},Z_{2,2})\in T$. We begin by constructing a path in $T$ from $P$ to $Q=(X_{1,1},0,X_{2,2},Y_{1,1},0,Y_{2,2},Z_{1,1},0,Z_{2,2})$. Let $t\in [0,1]$ and set 
$$
\gamma_1(t)=(X_{1,1},X_{1,2}(1-t),X_{2,2},Y_{1,1},Y_{1,2}(1-t),Y_{2,2},Z_{1,1},Z_{2,1}(1-t),Z_{2,2}).
$$
Then $\gamma_1(0)=P$, $\gamma_1(1)=Q$, and $\gamma_1(t)$ satisfies the equations (\ref{O5eq}) for all $t\in [0,1]$. Therefore, $\gamma_1(t)$ is a path from $P$ to $Q$ that lies in $T$.

Next, we construct a path in $T$ from $Q$ to $R=(1,0,1,1,0,1,1,0,1)$, completing the proof. Since $\mathbb{C}\setminus \{0\}$ is path connected, there exist paths in $\mathbb{C}\setminus \{0\}$ sending $X_{1,1}$, $Y_{1,1}$, and $Z_{2,2}$ each to $1$. Denote these paths by $X_{1,1}(t)$, $Y_{1,1}(t)$, and $Z_{2,2}(t)$ respectively. We use these to define a path from $Q$ to $R$ in $T$:
$$
\gamma_2(t)=\left( X_{1,1}(t), 0, \frac{1}{Y_{1,1}(t)\cdot Z_{2,2}(t)}, Y_{1,1}(t), 0, \frac{1}{X_{1,1}(t)\cdot Z_{2,2}(t)}, \frac{1}{X_{1,1}(t)\cdot Y_{1,1}(t)},0,Z_{2,2}(t)\right).
$$
Using the equations (\ref{O5eq}), we see that $\gamma_2(0)=Q$, $\gamma_2(1)=R$, and $\gamma_2(t)$ lies in $T$ for all $t\in [0,1]$.
\end{proof}

\noindent We conclude that for all orbits $O\neq O_6$, there is a unique simple object in $\textnormal{mod}_{\tn{GL}}(\D)$ with support $\overline{O}$. We will now show that there are two simples with full support.

\begin{lemma}
The isotropy group of the point $v=a_1\otimes b_1\otimes c_1+a_2\otimes b_2\otimes c_2$ has two connected components:
\begin{itemize}
\item
$\left\{ \left( \left( {\begin{array}{cc}
   x_{1,1} & 0 \\
   0 & x_{2,2}
  \end{array} } \right),
\left( {\begin{array}{cc}
   y_{1,1} & 0 \\
   0 & y_{2,2}
  \end{array} } \right),
 \left( {\begin{array}{cc}
   z_{1,1} & 0 \\
   0 & z_{2,2} 
  \end{array} } \right)\right)
\mid x_{1,1}y_{1,1}z_{1,1}=1, x_{2,2}y_{2,2}z_{2,2}=1\right\}$\\

\item
$\left\{ \left(\left( {\begin{array}{cc}
   0 & x_{1,2} \\
   x_{2,1} & 0
  \end{array} } \right),
\left( {\begin{array}{cc}
   0 & y_{1,2} \\
   y_{2,1} & 0
  \end{array} } \right),
 \left( {\begin{array}{cc}
   0 & z_{1,2} \\
   z_{2,1} & 0
  \end{array} } \right)\right)
\mid x_{1,2}y_{1,2}z_{1,2}=1, x_{2,1}y_{2,1}z_{2,1}=1\right\}$
\end{itemize}
In particular, the component group corresponding to the dense orbit $O_6$ is $\mathbb{Z}/2\mathbb{Z}$.
\end{lemma}

\begin{proof}
Similar to the proofs above, write $g\cdot v=\sum_{i,j,k}f_{i,j,k}a_i\otimes b_j\otimes c_k$. The condition that $g$ is in the isotropy of $v$ is eight equations: $f_{1,1,1}=1$, $f_{2,2,2}=1$, and $f_{i,j,k}=0$ otherwise. Considering the two cases $x_{1,1}=0$, $x_{1,1}\neq 0$ yields that the isotropy for $v$ is the disjoint union of the two sets above. In both cases, the matrix $Z$ is determined by the matrices $X$ and $Y$. We conclude that both sets above are isomorphic to $(\mathbb{C}^{\ast})^4$, where two copies of $\mathbb{C}^{\ast}$ correspond to $X$ and two copies of $\mathbb{C}^{\ast}$ correspond to $Y$. Therefore, the two subsets of $\tn{GL}$ above are connected, and the isotropy subgroup for $v$ has two connected components.
\end{proof}

\noindent By the above, conclude that there are eight simple objects in $\textnormal{mod}_{\tn{GL}}(\D)$. For orbits $O\neq O_6$, the simple object is the intersection homology $\D$-module $\mathcal{L}(O,V)$. These are denoted by $D$, with a subscript denoting which orbit they correspond to. The simples with full support are $D_6=S=\mathcal{L}(O_6,V)$ and $G_6=\mathcal{L}(O_6,\mathcal{G},V)$, where $\mathcal{G}$ is a nontrivial equivariant local system on $O_6$.

\subsection{Witness weights for the simple $\mathcal{D}$-modules}\label{Witness}
In this subsection, we describe the composition factors of $S_h$ and $S_h\cdot \sqrt{h}$, obtaining the witness weights for the simple objects along the way. In addition, we complete the proof of the Classification of Simple Modules. By \cite[Proposition 4.9]{lHorincz2018categories}, the filtrations of $S_h$ and $S_h\cdot \sqrt{h}$ are dictated by the Bernstein-Sato polynomial of $h$ (see \cite{sato1980micro} or \cite[Example 2.9]{lHorincz2018decompositions}):
\begin{equation}\label{bFunction}
b_h(s)=\left(s+1\right)\left(s+3/2\right)^2\left(s+2\right).
\end{equation}
For any $r\in \C$, we consider the $\D$-module $\langle h^r\rangle_{\D}$ that is the $\D$-submodule of $S_h\cdot h^r$ generated by $h^r$. By \cite[Proposition 4.9]{lHorincz2018categories}, the $\D$-module $\langle h^r\rangle_{\D}/\langle h^{r+1}\rangle_{\D}$ is nonzero if and only if $r$ is a root of the Bernstein-Sato polynomial $b_h(s)$. By (\ref{bFunction}) we obtain the following:
\begin{equation}\label{filtrations}
0\subsetneq S\subsetneq \langle h^{-1}\rangle_{\D}\subsetneq \langle h^{-2}\rangle_{\D}= S_h,\;\;\;\;0\subsetneq \langle h^{-1/2}\rangle_{\D}\subsetneq \langle h^{-3/2}\rangle_{\D}=S_h\cdot \sqrt{h}.
\end{equation}
We summarize here the results that we prove in this subsection. Given a simple module $M$, write $\lambda(M)$ for the set of witness weights. 

\begin{theorem}\label{witnessWeights}
The composition factors of $S_h$ are $S$, $E$, and $D_5$, each with multiplicity one, and the composition factors of $S_h\cdot \sqrt{h}$ are $G_6$, $D_{1,2,,2}$, $D_{2,1,2}$, $D_{2,2,1}$, and $D_1$, each with multiplicity one. More precisely, the inclusions $S\subsetneq \langle h^{-1}\rangle_{\D}$, $\langle h^{-1}\rangle_{\D}\subsetneq S_h$, and $\langle h^{-1/2}\rangle_{\D}\subsetneq S_h\cdot \sqrt{h}$ are non-split and:
$$
\langle h^{-1}\rangle_{\D}/S\cong D_5,\;\; S_h/\langle h^{-1}\rangle_{\D}\cong E,\;\;\langle h^{-1/2}\rangle_{\D}\cong G_6,\;\;\textnormal{and there is a non-split short exact sequence}
$$
$$
0\longrightarrow D_{1,2,2}\oplus D_{2,1,2}\oplus D_{2,2,1}\longrightarrow (S_h\cdot \sqrt{h})/G_6\longrightarrow D_1\longrightarrow 0.
$$
We have the following witness weights (using Convention \ref{convention}):
$$
(0,0)^3\in \lambda(S),\;\;\;  (1,1)^3\in \lambda(G_6),\;\;\; (2,2)^3\in \lambda(D_5),\;\;\;  (3,3)^3\in \lambda(D_1),\;\;\; (4,4)^3\in \lambda(E),
$$
$$
(3,1)\times (2,2)^2\in \lambda(D_{1,2,2}),\;\; (2,2)\times (3,1)\times (2,2)\in \lambda(D_{2,1,2}),\;\;(2,2)^2\times (3,1)\in \lambda(D_{2,2,1}).
$$
The characteristic cycles of the simple modules are described as follows:
$$
\operatorname{charC}(D_0)=[T^{\ast}_{O_0}V],\;\;\;\; \operatorname{charC}(D_1)=[T^{\ast}_{O_0}V]+[\overline{T^{\ast}_{O_1}V}],\;\;\;\;\operatorname{charC}(D_{i,j,k})=[\overline{T^{\ast}_{O_{i,j,k}}V}],
$$
$$
\operatorname{charC}(D_5)=[\overline{T^{\ast}_{O_1}V}]+[\overline{T^{\ast}_{O_{1,2,2}}V}]+[\overline{T^{\ast}_{O_{2,1,2}}V}]+[\overline{T^{\ast}_{O_{2,2,1}}V}]+[\overline{T^{\ast}_{O_5}V}],
$$
$$
\operatorname{charC}(S)=[T^{\ast}_{V}V],\;\;\;\;\operatorname{charC}(G_6)=[\overline{T^{\ast}_{O_5}V}]+[T^{\ast}_{V}V].
$$
\end{theorem}

\noindent In Sections \ref{RepTheory} and \ref{D}, we discussed the $\tn{GL}$-equivariant structure of $S$, $E$, $S_h$, $S_h\cdot \sqrt{h}$, and in Section \ref{PrelimWit} we found weights that appear in $D_{i,j,k}$ for $(i,j,k)=(1,2,2)$, $(2,1,2)$, and $(2,2,1)$. Now we study the weights that appear in $D_1$.

\begin{lemma}\label{witD1}
The multiplicity of $\bS_{(3,3)}A\otimes \bS_{(3,3)}B\otimes \bS_{(3,3)}C$ in $D_1$ is one. For $a\leq 2$, the multiplicity of $\bS_{(a,a)}A\otimes \bS_{(a,a)}B\otimes \bS_{(a,a)}C$ in $D_1$ is zero. Finally, the representations $\bS_{(3,1)}A\otimes \bS_{(2,2)}B\otimes \bS_{(2,2)}C$, $\bS_{(2,2)}A\otimes \bS_{(3,1)}B\otimes \bS_{(2,2)}C$, and $\bS_{(2,2)}A\otimes \bS_{(2,2)}B\otimes \bS_{(3,1)}C$ do not appear in $D_1$.
\end{lemma}

\begin{proof}
Let $(\gamma,\delta,\sigma)\in \Lambda$ be one of the triples of dominant weights in the statement of the lemma. Use the notation from Section \ref{PrelimWit}. By Lemma \ref{notE}, to prove the assertion it suffices to show that 
$$
\left\langle \left[ \chi\left(\int_{\pi}\mathcal{O}_{O_1}\right)\right], [\bS_{\gamma}A\otimes \bS_{\delta}B\otimes \bS_{\sigma}C]\right\rangle=
\begin{cases}
1 & \tn{ if $(\gamma,\delta,\sigma)=((3,3),(3,3),(3,3))$},\\
0 & \tn{ otherwise}.
\end{cases}
$$
By Lemma \ref{eulerForm} we have that the multiplicity of $[\bS_{\gamma}A\otimes \bS_{\delta}B\otimes \bS_{\sigma}C]$ in $[ \chi(\int_{\pi}\mathcal{O}_{O_1})]$ is equal to
$$
\lim_{r\to \infty}\left( -\sum_{I,J,K\in \binom{[2]}{1}}\langle [\textnormal{Sym}(V)], [\bS_{(\gamma_1-4,\gamma_2-4)(r,I)}A\otimes \bS_{(\delta_1-4,\delta_2-4)(r,J)}B\otimes \bS_{(\sigma_1-4,\sigma_2-4)(r,K)}C]\rangle\right).
$$
Since $\gamma_2-4<0$, we have 
$$
 \langle [\textnormal{Sym}(V)], [\bS_{(\gamma_1-4,\gamma_2-4)(r,I)}A\otimes \bS_{(\delta_1-4,\delta_2-4)(r,J)}B\otimes \bS_{(\sigma_1-4,\sigma_2-4)(r,K)}C]\rangle\neq 0 \textnormal{  only if  } I=J=K=\{2\}.
$$
Using the convention (\ref{bott}), we obtain that the multiplicity of $[\bS_{\gamma}A\otimes \bS_{\delta}B\otimes \bS_{\sigma}C]$ in $[\chi(\int_{\pi}\mathcal{O}_{O_1})]$ is given by
\begin{equation}\label{limitMult}
\lim_{r\to \infty} \langle [\textnormal{Sym}(V)], [\bS_{(\gamma_2+r-5,\gamma_1-3)}A\otimes \bS_{(\delta_2+r-5,\delta_1-3)}B\otimes \bS_{(\sigma_2+r-5,\sigma_1-3)}C]\rangle.
\end{equation}
When $(\gamma,\delta,\sigma)\neq ((3,3),(3,3),(3,3))$, then one of the following holds: $\gamma_1-3<0$, $\delta_1-3<0$, or $\sigma_1-3<0$. Thus, in that scenario, every term in (\ref{limitMult}) is zero, implying that the limit is zero. Therefore, we have proven all assertions in the statement of the lemma except the assertion about $\bS_{(3,3)}A\otimes \bS_{(3,3)}B\otimes \bS_{(3,3)}C$. 

To complete the proof, let $(\gamma,\delta,\sigma)= ((3,3),(3,3),(3,3))$. We need to show that the multiplicity of $\bS_{(\gamma_2+r-5,\gamma_1-3)}A\otimes \bS_{(\delta_2+r-5,\delta_1-3)}B\otimes \bS_{(\sigma_2+r-5,\sigma_1-3)}C=\bS_{(r-2,0)}A\otimes \bS_{(r-2,0)}B\otimes \bS_{(r-2,0)}C$ in $\textnormal{Sym}(V)$ is one for $r\gg 0$. Dualizing (\ref{charS}) yields the decomposition of $\tn{Sym}(V)=S^{\ast}$ into irreducibles, and we will use Lemma \ref{raicuFormula} to compute the desired multiplicities. Using the notation of that lemma, let $(\lambda,\mu,\nu)=((r-2,0),(r-2,0),(r-2,0))$. Then $d=r-2$, $f_{\lambda,\mu,\nu}=0$, and $e_{\lambda,\mu,\nu}=0$. For $r\gg 0$, we have that $e_{\lambda,\mu,\nu}< d-1$ and $e_{\lambda,\mu,\nu}\geq 2f_{\lambda,\mu,\nu}$. Since $e_{\lambda,\mu,\nu}$ is even, Lemma \ref{raicuFormula} and (\ref{charS}) imply that the multiplicity of $\bS_{(r-2,0)}A\otimes \bS_{(r-2,0)}B\otimes \bS_{(r-2,0)}C$ in $\textnormal{Sym}(V)$ for $r\gg 0$ is one, as claimed.
\end{proof}

\noindent Using the information of Lemma \ref{witD1}, we now describe the composition factors of $S_h$ and $S_h\cdot \sqrt{h}$:

\begin{lemma}\label{compFacs1}
The modules $S$ and $E$ are not composition factors of $S_h\cdot \sqrt{h}$. The modules $D_{1,2,2}$, $D_{2,1,2}$, $D_{2,2,1}$, $D_1$, $\mathcal{F}(D_1)$ are not composition factors of $S_h$.
\end{lemma}

\begin{proof}
By Lemma \ref{multSh}, the multiplicities of $\bS_{(0,0)}A\otimes \bS_{(0,0)}B\otimes \bS_{(0,0)}C$ and $\bS_{(4,4)}A\otimes \bS_{(4,4)}B\otimes \bS_{(4,4)}C$ in $S_h\cdot \sqrt{h}$ are both zero. Since $S$ contains the subrepresentation $\bS_{(0,0)}A\otimes \bS_{(0,0)}B\otimes \bS_{(0,0)}C$, it is not a composition factor of $S_h\cdot \sqrt{h}$. Similarly, by Lemma \ref{notE}, the simple module $E$ is not a composition factor of $S_h\cdot \sqrt{h}$.

To prove the second assertion, recall that by Lemma \ref{multDet}, we have 
$$
\langle [D_{1,2,2}], [\bS_{(3,1)}A\otimes \bS_{(2,2)}B\otimes \bS_{(2,2)}C]\rangle=1,\;\;\langle [D_{2,1,2}], [\bS_{(2,2)}A\otimes \bS_{(3,1)}B\otimes \bS_{(2,2)}C]\rangle=1,
$$
$$
\langle [D_{2,2,1}], [\bS_{(2,2)}A\otimes \bS_{(2,2)}B\otimes \bS_{(3,1)}C]\rangle=1.
$$
Using Lemma \ref{multSh}, these weights do no appear in $S_h$. Thus, $D_{1,2,2}$, $D_{2,1,2}$, $D_{2,2,1}$ cannot be composition factors of $S_h$. Finally, by Lemma \ref{witD1}, the simple $D_1$ contains the weight $(3,3)^3$, and thus $\mathcal{F}(D_1)$ contains the weight $(1,1)^3$. Again, by Lemma \ref{multSh}, the result follows.
\end{proof}

Since $\langle h^{-1/2}\rangle_{\D}$ has full support, it must contain $S$ or $G_6$ as a submodule. By Lemma \ref{compFacs1} it follows that $G_6\subseteq \langle h^{-1/2}\rangle_{\D}$. We now prove that this is in fact an equality. 

\begin{lemma}\label{Bsub}
The multiplicity of $\bS_{(1,1)}A\otimes \bS_{(1,1)}B\otimes \bS_{(1,1)}C$ in $G_6$ is one, $\mathcal{F}(D_1)\cong G_6$, and $G_6\cong \langle h^{-1/2}\rangle_{\D}$.
\end{lemma}

\begin{proof}
We begin by proving the first claim. By Lemma \ref{multSh}, the module $S_h\cdot \sqrt{h}$ contains the representation $\bS_{(1,1)}A\otimes \bS_{(1,1)}B\otimes \bS_{(1,1)}C$ with multiplicity one. Thus, $S_h\cdot \sqrt{h}$ must have a composition factor with the weight $(1,1)^3$. By Lemma \ref{raicuFormula}, Lemma \ref{notE}, Lemma \ref{multDet}, and Lemma \ref{witD1}, the simples $D_0=E$, $D_1$, $D_{i,j,k}$, and $S$ do not contain the weight $(1,1)^3$, so we conclude that the composition factor of $S_h\cdot \sqrt{h}$ with this weight must be $D_5$ or $G_6$. Since $\bS_{(1,1)}A\otimes \bS_{(1,1)}B\otimes \bS_{(1,1)}C$ has multiplicity one in $S_h\cdot \sqrt{h}$, it suffices to show that the multiplicity of $\bS_{(1,1)}A\otimes \bS_{(1,1)}B\otimes \bS_{(1,1)}C$ in $D_5$ is zero. By Lemma \ref{IHmod}, and the \v{C}ech cohomology description of local cohomology, it follows that $D_5\subset S_h/S$ (i.e. it is a composition factor of $S_h$). By Lemma \ref{multSh}, conclude that $G_6$ contains the weight $(1,1)^3$, and therefore $\mathcal{F}(D_1)\cong G_6$. The third assertion follows, since $h^{-1/2}$ has weight $(1,1)^3$, again by Lemma \ref{multSh}.
\end{proof}

\begin{lemma}\label{D1notSub}
There is a non-split short exact sequence
$$
0\longrightarrow D_{1,2,2}\oplus D_{2,1,2}\oplus D_{2,2,1}\longrightarrow (S_h\cdot \sqrt{h})/G_6\longrightarrow D_1\longrightarrow 0.
$$
\end{lemma}

\begin{proof}
We begin by showing that $D_1$ is not a submodule of $(S_h\cdot \sqrt{h})/G_6=\langle h^{-3/2}\rangle_{\D}/\langle h^{-1/2}\rangle_{\D}$. By the proof of \cite[Proposition 4.9]{lHorincz2018categories}, it follows that $\langle h^{-3/2}\rangle_{\D}/\langle h^{-1/2}\rangle_{\D}$ has a unique $\D$-simple quotient, containing the representation $S_{(3,3)}A\otimes S_{(3,3)}B\otimes S_{(3,3)}C$. Since $D_1$ is the only simple containing this representation, we conclude that $D_1$ is a quotient of $(S_h\cdot \sqrt{h})/G_6=\langle h^{-3/2}\rangle_{\D}/\langle h^{-1/2}\rangle_{\D}$. By Lemma \ref{multSh}, Lemma \ref{witD1}, and Lemma \ref{Bsub}, we conclude that $D_{1,2,2}$, $D_{2,1,2}$, and $D_{2,2,1}$ are also composition factors of $(S_h\cdot \sqrt{h})/G_6$. If $D_1$ were a submodule of $(S_h\cdot \sqrt{h})/G_6$, then it would be a direct summand. In particular, $D_1$ and one of $D_{i,j,k}$ would both be quotients of $(S_h\cdot \sqrt{h})/G_6$, contradicting the second assertion of \cite[Proposition 4.9]{lHorincz2018categories}. Therefore, $D_1$ is not a submodule of $(S_h\cdot \sqrt{h})/G_6$. It follows that for some $(i,j,k)=(1,2,2)$, $(2,1,2)$, $(2,2,1)$, the simple $D_{i,j,k}$ is a submodule of $(S_h\cdot \sqrt{h})/G_6$. The natural $\mathbb{Z}/3\mathbb{Z}$ action on $V$ extends to an action on $S_h\cdot \sqrt{h}$ preserving $G_6$ and permuting $D_{i,j,k}$, so we have that $D_{i,j,k}$ is a submodule of $(S_h\cdot \sqrt{h})/G_6$ for all $(i,j,k)$. Since the $D_{i,j,k}$'s are simple, it follows that $D_{1,2,2}\oplus D_{2,1,2}\oplus D_{2,2,1}$ is a submodule of $(S_h\cdot \sqrt{h})/G_6$. Therefore, the short exact sequence in the statement of the lemma exists, and is non-split.
\end{proof}

\noindent Combining the three previous lemmas and their proofs, we complete the proof of Theorem \ref{witnessWeights}:

\begin{proof}
[Proof of Theorem \ref{witnessWeights}] We begin by proving the claim about the composition factors and filtration of $S_h$. By (\ref{filtrations}), the $\D$-module $S_h$ has length greater than or equal to three. By Lemma \ref{compFacs1}, its composition factors are among $S$, $E$, and $D_5$. Since $h^{-1}$ is of weight $(2,2)^3$, the submodule $\langle h^{-1}\rangle_{\D}\subsetneq S_h$ has a simple quotient containing this weight. Using Lemma \ref{notE} and the fact that $S$ only has weights $(\lambda,\mu,\nu)$ with $\lambda_i\leq 0$, $\mu_i\leq 0$, and $\nu_i\leq 0$ for $i=1,2$, it follows that $D_5$ has the weight $(2,2)^3$, and $\langle h^{-1}\rangle$ surjects onto $D_5$. By a similar argument, we have also that $E$ is a quotient of $S_h$. Since the weights $(0,0)^3$, $(2,2)^3$, and $(4,4)^3$ each appear in $S_h$ with multiplicity one by Lemma \ref{multSh}, we obtain that $S$, $E$, and $D_5$ are each composition factors of $S_h$ of multiplicity one, $\langle h^{-1}\rangle_{\D}/S\cong D_5$, and $S_h/\langle h^{-1}\rangle_{\D}\cong E$.

Next, using Lemma \ref{Bsub} and Lemma \ref{D1notSub}, the assertions about the filtration and composition factors of $S_h\cdot \sqrt{h}$ are immediate. Since the module $\langle h^{s}\rangle_{\D}$ contains the weight $(-2s,-2s)^3$ for all $s\in \mathbb{Q}$, it follows that each simple contains the above claimed witness weights. These weights are unique to their respective simple module by Lemma \ref{multSh}.

We now calculate the characteristic cycles of the simple modules. The descriptions of $\operatorname{charC}(S)$ and $\operatorname{charC}(D_0)$ are standard, and only included in the statement of Theorem \ref{witnessWeights} for completeness. Next, via each of the three flattenings of $V$ to a space of $4\times 2$ matrices, the orbit closures $\overline{O_{i,j,k}}$ are identified with matrices of rank $\leq 1$, so it is known that $D_{i,j,k}$ have irreducible characteristic cycle \cite{raicu2016characters}, as claimed. Since the hypersurface $\overline{O_5}$ is projective dual to $\overline{O_1}$, and $\mathcal{F}(G_6)=D_1$ by Lemma \ref{Bsub}, it follows from \cite[Section 4.3]{lHorincz2018categories} that the characteristic cycles of $D_1$ and $G_6$ are as asserted.  

It remains to determine the characteristic cycle of $D_5$. Since $S_h\cdot \sqrt{h}$ has composition factors $D_1$, $D_{1,2,2}$, $D_{2,1,2}$, $D_{2,2,1}$, and $G_6$, each with multiplicity one, we have that the characteristic cycle of $S_h\cdot \sqrt{h}$ is given by
$$
\operatorname{charC}(S_h\cdot \sqrt{h})=[T^{\ast}_{O_0}V]+[\overline{T^{\ast}_{O_1}V}]+[\overline{T^{\ast}_{O_{1,2,2}}V}]+[\overline{T^{\ast}_{O_{2,1,2}}V}]+[\overline{T^{\ast}_{O_{2,2,1}}V}]+[\overline{T^{\ast}_{O_5}V}]+[T^{\ast}_{V}V].
$$
By the characteristic cycle of $G_6$, we see that the singular locus of $G_6$ is $\overline{O_5}$, so that $\operatorname{charC}(S_h)$ is equal to $\operatorname{charC}(S_h\cdot \sqrt{h})$ \cite[Theorem 3.2]{MR833194} (see also \cite[Lemma 1.11]{lHorincz2021holonomic}). In particular,
$$
\operatorname{charC}(D_5)=\operatorname{charC}(S_h)-\operatorname{charC}(S)-\operatorname{charC}(D_0)=[\overline{T^{\ast}_{O_1}V}]+[\overline{T^{\ast}_{O_{1,2,2}}V}]+[\overline{T^{\ast}_{O_{2,1,2}}V}]+[\overline{T^{\ast}_{O_{2,2,1}}V}]+[\overline{T^{\ast}_{O_5}V}],
$$
as required to complete the proof of Theorem \ref{witnessWeights}.
\end{proof}

\begin{prop}\label{fourProp}
The holonomic duality functor fixes all of the simple modules. The Fourier transform swaps the modules in the two pairs $(S,E)$, $(G_6,D_1)$, and all other simples are fixed. 
\end{prop}

\begin{proof}
The holonomic duality functor $\mathbb{D}$ sends the simple module $M$ corresponding to the local system $\mathcal{M}$ on an orbit $O$ to the simple module corresponding to the dual local system $\mathcal{M}^{\ast}$. Since all the simple modules except $G_6$ correspond to the trivial local systems on their respective orbits, the first statement follows. The second assertion follows from the witness weight computations in Theorem \ref{witnessWeights} and the definition of the Fourier transform (\ref{fourier}) in Section \ref{D}.
\end{proof}

\subsection{The quiver structure of the category $\tn{mod}_{\tn{GL}}(\D)$}

For the remainder of the section, we prove the Theorem on the Quiver Structure. We begin by proving a couple of lemmas about which nontrivial extensions are possible between the simple objects. We refer to the following short exact sequences coming from the filtration of $S_h\cdot \sqrt{h}$:
\begin{equation}\label{sesG}
0\longrightarrow G_6\longrightarrow F\longrightarrow D_{1,2,2}\oplus D_{2,1,2}\oplus D_{2,2,1}\longrightarrow 0,
\end{equation}
\begin{equation}\label{sesF}
0\longrightarrow F\longrightarrow S_h\cdot \sqrt{h}\longrightarrow D_1\longrightarrow 0.
\end{equation}
It is important to note that by Lemma \ref{Bsub} and \cite[Lemma 2.4]{lHorincz2017equivariant}, $S_h$ is the injective hull of $S$ and $S_h\cdot \sqrt{h}$ is the injective hull of $G_6$ in $\textnormal{mod}_{\tn{GL}}(\D)$. 

\begin{lemma}\label{extBD}
For all $(i,j,k)=(1,2,2), (2,1,2), (2,2,1)$, we have the following in $\textnormal{mod}_{\tn{GL}}(\D)$:
\begin{equation}
\dim_{\C}\textnormal{Ext}_{\D}^1(D_{i,j,k},G_6)=1,\;\; \tn{and}\;\; \textnormal{Ext}_{\D}^1(D_{i,j,k},D_5)=0.
\end{equation}
\end{lemma}

\begin{proof}
We start by proving the first assertion. Applying $\textnormal{Hom}(D_{i,j,k},-)$ to (\ref{sesF}) yields $\dim_{\C}\textnormal{Ext}^1_{\D}(D_{i,j,k},F)=0$, since $S_h\cdot \sqrt{h}$ is injective and $\textnormal{Hom}(D_{i,j,k},D_1)=0$. Now applying $\textnormal{Hom}(D_{i,j,k},-)$ to (\ref{sesG}) yields the desired result since $\dim_{\C}\textnormal{Hom}(D_{i,j,k},D_{1,2,2}\oplus D_{2,1,2}\oplus D_{2,2,1})=1$ and $\textnormal{Ext}^1_{\D}(D_{i,j,k},F)=0$.

To prove that $\textnormal{Ext}^1_{\D}(D_{i,j,k},D_5)=0$ we will show that there are no nontrivial extensions between these two simples in the full subcategory $\tn{mod}_{\tn{GL}}^{\overline{O_5}}(\D)$ of modules with support contained in $\overline{O_5}$. Using bullet (5) in Section \ref{PrelimQuiver}, it suffices to show that for all $(i,j,k)$, the module $D_{i,j,k}$ is not a composition factor of the injective hull of $D_5$ in $\tn{mod}_{\tn{GL}}^{\overline{O_5}}(\D)$. Let $Z=\overline{O_5}\setminus O_5$ and let $j$ be the open immersion $j:V\setminus Z\hookrightarrow V$. By \cite[Lemma 2.4]{lHorincz2017equivariant}, the module $j_{\ast}j^{\ast}D_5$ is the injective hull of $D_5$ in $\textnormal{mod}_{\tn{GL}}^{\overline{O_5}}(\D)$. By (\ref{locSES}) with $M=D_5$, it suffices to show that $D_{i,j,k}$ is not a composition factor of $H^1_Z(D_5)$ for all $(i,j,k)$. Consider the Mayer-Vietoris sequence coming from the fact that $\overline{O_{i,j,k}}\cap \overline{O_{p,q,r}}=\overline{O_1}$ when $(i,j,k)\neq (p,q,r)$:
$$
\cdots \longrightarrow H^i_{\overline{O_1}}(D_5)\longrightarrow H^i_{\overline{O_{1,2,2}}}(D_5)\oplus H^i_{\overline{O_{2,1,2}}}(D_5)\longrightarrow H^i_{\overline{O_{1,2,2}}\cup \overline{O_{2,1,2}}}(D_5)\longrightarrow \cdots
$$
Using Proposition \ref{locD5}, we obtain that $H^1_{\overline{O_{1,2,2}}\cup \overline{O_{2,1,2}}}(D_5)=0$ (the proof of Proposition \ref{locD5} only relies on Theorem \ref{witnessWeights} and Proposition \ref{locS}, and does not use the Theorem on the Quiver Structure).
Now consider the Mayer-Vietoris sequence (again, coming from the fact that $\overline{O_{i,j,k}}\cap \overline{O_{p,q,r}}=\overline{O_1}$ when $(i,j,k)\neq (p,q,r)$):
$$
\cdots \longrightarrow H^i_{\overline{O_1}}(D_5)\longrightarrow H^i_{\overline{O_{1,2,2}}\cup \overline{O_{2,1,2}}}(D_5)\oplus H^i_{\overline{O_{2,2,1}}}(D_5)\longrightarrow H^i_Z(D_5)\longrightarrow \cdots
$$
By Proposition \ref{locD5}, and the fact that $H^1_{\overline{O_{1,2,2}}\cup \overline{O_{2,1,2}}}(D_5)=0$, we get $H^1_Z(D_5)=0$, completing the proof.
\end{proof}

\noindent After applying the holonomic duality functor $\mathbb{D}$, the previous Lemma yields: $\dim_{\C}\textnormal{Ext}^1(G_6,D_{i,j,k})=1$, and $\textnormal{Ext}^1(D_5,D_{i,j,k})=0$.

\begin{lemma}\label{newExt}
In $\tn{mod}_{\tn{GL}}(\D)$, we have $\textnormal{Ext}_{\D}^1(D_1,G_6)=0$. 
\end{lemma}

\begin{proof}
Applying $\tn{Hom}(D_1,-)$ to the short exact sequence 
$$
0\longrightarrow G_6\longrightarrow S_h\cdot \sqrt{h}\longrightarrow (S_h\cdot \sqrt{h})/G_6\longrightarrow 0,
$$
yields the long exact sequence of $\tn{Ext}^{\bullet}_{\D}(D_1,-)$. Since $S_h\cdot \sqrt{h}$ is injective and $\tn{Hom}(D_1, (S_h\cdot \sqrt{h})/G_6)=0$, it follows that $\textnormal{Ext}_{\D}^1(D_1,G_6)=0$, as required.
\end{proof}

\begin{lemma}\label{noPath}
For all $(i,j,k)$ and $(p,q,r)$ we have $\tn{Ext}^1_{\D}(D_{i,j,k},D_{p,q,r})=0$.
\end{lemma}

\begin{proof}
Suppose for contradiction that there exists a nontrivial extension of $D_{i,j,k}$ by $D_{p,q,r}$ in $\tn{mod}_{\tn{GL}}(\D)$. By bullet (3) in Section \ref{PrelimQuiver}, such an extension would give an arrow  from $d_{i,j,k}$ to $d_{p,q,r}$ in the quiver with relations corresponding to the full subcategory $\tn{mod}^Y_{\tn{GL}}(\D)$, where $Y=\overline{O_{i,j,k}}\cup\overline{O_{p,q,r}}$.  By \cite[Corollary 3.9]{lHorincz2018categories} and Theorem \ref{witnessWeights}, there are no nontrivial paths from $d_{i,j,k}$ to $d_{p,q,r}$ in that quiver, yielding a contradiction.
\end{proof}

The following statement was suggested to us by Andr\'{a}s L\H{o}rincz.

\begin{lemma}\label{correctionLemma}
For each $(i,j,k)$, let $P_{i,j,k}$ denote the projective cover of $D_{i,j,k}$ in 	$\tn{mod}_{\tn{GL}}(\D)$. For all $(p,q,r)$, the conormal variety $[\overline{T^{\ast}_{O_{p,q,r}}V}]$ appears with multiplicity one in the characteristic cycle $\operatorname{charC}(P_{i,j,k})$.
\end{lemma}

\begin{proof}
By symmetry, it suffices to prove the result for $(i,j,k)=(1,2,2)$. Let $V=\bS_{(3,1)}A\otimes \bS_{(2,2)}B\otimes \bS_{(2,2)}C$, and let $P(V)$ be the projective module of \cite[Section 2.1]{lHorincz2018categories}. Since $V$ is a witness weight space for $D_{1,2,2}$ with multiplicity one, it follows that $P_{1,2,2}\cong P(V)$ is the projective cover of $D_{1,2,2}$ \cite[Lemma 2.1(a), Proposition 2.7]{lHorincz2018categories}.	

Let $\mathfrak{g}=\mathfrak{gl}(A)\times \mathfrak{gl}(B)\times \mathfrak{gl}(C)$ be the Lie algebra of $\textnormal{GL}$, with $\mathbb{C}$-basis $X_{i,j}$, $Y_{i,j}$, $Z_{i,j}$ ($1\leq i,j\leq 2$), and universal enveloping algebra $U(\mathfrak{g})$. Given a highest weight vector $v$ of $V$, we have the following isomorphism of $\D$-modules \cite[Equation (2.5)]{lHorincz2018categories}:
\begin{equation}\label{gensP}
P_{1,2,2}\cong \frac{\mathcal{D}}{\mathcal{D}\langle\operatorname{Ann}_{U(\mathfrak{g})}v\rangle},
\end{equation}
where $\mathfrak{g}$ maps to the Weyl algebra $\D=\C\langle x_{i,j,k},\partial_{i,j,k}\mid 1\leq i,j,k\leq 2\rangle$ via 
$$
X_{i,j} \mapsto \sum_{1\leq k,l\leq 2} x_{i,k,l}\partial_{j,k,l},\;\;\;\;\; Y_{i,j}\mapsto \sum_{1\leq k,l\leq 2} x_{k,i,l}\partial_{k,j,l},\;\;\;\;\; Z_{i,j}\mapsto \sum_{1\leq k,l\leq 2} x_{k,l,i}\partial_{k,l,j}.
$$
Using (\ref{gensP}) and the description of the generators of $\operatorname{Ann}_{U(\mathfrak{g})}v$ \cite[Theorem 21.4]{MR0323842}, we calculate the ideal of the characteristic variety of $P_{1,2,2}$ using Macaulay2 \cite{M2}:

\begin{verbatim}
loadPackage "Dmodules"
S=QQ[x_(1,1,1)..x_(2,2,2)];
W=makeWeylAlgebra S;

X = (i,j) -> (
x_(i,1,1)*dx_(j,1,1)+x_(i,1,2)*dx_(j,1,2)+x_(i,2,1)*dx_(j,2,1)+x_(i,2,2)*dx_(j,2,2));
Y = (i,j) -> (
x_(1,i,1)*dx_(1,j,1)+x_(1,i,2)*dx_(1,j,2)+x_(2,i,1)*dx_(2,j,1)+x_(2,i,2)*dx_(2,j,2));
Z = (i,j) -> (
x_(1,1,i)*dx_(1,1,j)+x_(1,2,i)*dx_(1,2,j)+x_(2,1,i)*dx_(2,1,j)+x_(2,2,i)*dx_(2,2,j));

I=ideal(X(1,1)+1,X(1,2),(X(2,1))^3,X(2,2)+3);
I=I+ideal(Y(1,1)+2,Y(1,2),Y(2,1),Y(2,2)+2);
I=I+ideal(Z(1,1)+2,Z(1,2),Z(2,1),Z(2,2)+2);
J=charIdeal I;
\end{verbatim}
Using the above calculated characteristic ideal of $P_{1,2,2}$, we verify that $[\overline{T^{\ast}_{O_{1,2,2}}V}]$, $[\overline{T^{\ast}_{O_{2,1,2}}V}]$ and $[\overline{T^{\ast}_{O_{2,2,1}}V}]$ appear in the characteristic cycle of $P_{1,2,2}$.
\begin{verbatim}
sub(J, matrix{{1,0,0,1,0,0,0,0,0,0,0,0,1,0,0,-1}}) == 0
sub(J, matrix{{1,0,0,0,0,1,0,0,0,0,1,0,0,0,0,-1}}) == 0
sub(J, matrix{{1,0,0,0,0,0,1,0,0,1,0,0,0,0,0,-1}}) == 0
\end{verbatim}
where we used that $\overline{O_{1,2,2}}$, $\overline{O_{2,1,2}}$, and $\overline{O_{2,2,1}}$ flatten to $4\times 2$ matrices of rank $\leq 1$, and the description \cite{MR653906} of the conormal variety to a determinantal variety.

It remains to show that these conormal varieties appear with multiplicity one. Using notation from (\ref{gel}), we let $H$ be the subgroup of $\textnormal{GL}$ defined by $z_{2,1}=0$. Then the highest weight vector of $V$ is $H$-semi-invariant, so by \cite[Lemma 3.12]{lHorincz2018categories} and the proof of \cite[Proposition 3.14]{lHorincz2018categories}, it suffices to show that each $[\overline{T^{\ast}_{O_{p,q,r}}V}]$ has a dense $H$ orbit. Using Macaulay2, one can show that the dimension of the $H$-stabilizer of each of the representatives appearing in the Macaulay2 code above have dimension three. Since $H$ has dimension eleven, and each conormal has dimension eight, it follows that these orbits must be dense.
\end{proof}

We now prove the Theorem on Quiver Structure.

\begin{proof}
[Proof of Theorem on the Quiver Structure] Write $s$, $e$, $d_5$, $g_6$, and $d_{i,j,k}$ for the vertices in $(\mathcal{Q},\mathcal{I})$ corresponding to the simple objects $S$, $E$, $D_5$, $G_6$, and $D_{i,j,k}$ respectively. Recall the facts (1)-(5) in Section \ref{PrelimQuiver}. Since $S_h$ is the injective hull of $S$, by \cite[Lemma 4.1]{lHorincz2017equivariant} and Theorem \ref{witnessWeights} we have $\dim_{\C}\textnormal{Ext}^1_{\D}(D_5,S)=1$ and $\textnormal{Ext}^1_{\D}(E,S)=0$. Thus, by bullet (3) in Section \ref{PrelimQuiver}, there is a unique arrow $\psi_0$ from $d_5$ to $s$ and no arrow from $e$ to $s$. Applying the Fourier transform, it follows that there is a unique arrow $\varphi_1$ from $d_5$ to $e$, and no arrows from $s$ to $e$ (by bullet (4) in Section \ref{PrelimQuiver}). Using the holonomic duality functor, we obtain a unique arrow $\varphi_0$ from $s$ to $d_5$, and a unique arrow $\psi_1$ from $e$ to $d_5$ (again, by bullet (4)). Since $S_h$ is the injective hull of $S$, and $S_h$ has composition factors $S$, $D_5$, and $E$, each with multiplicity one, we obtain the relations $\varphi_0\psi_0,\psi_0\varphi_0,\varphi_1\psi_1,\psi_1\varphi_1$ (by bullet (5)).

Now we consider the vertices of $(\mathcal{Q},\mathcal{I})$ corresponding to the composition factors of $S_h\cdot \sqrt{h}$. By Lemma \ref{extBD}, we have $\dim_{\C}\textnormal{Ext}^1(D_{i,j,k},G_6)=1$ for all $(i,j,k)$. Thus, there are unique arrows $\alpha_{i,j,k}$ from $d_{i,j,k}$ to $g_6$ (by bullet (3)). Applying the Fourier transform, we obtain unique arrows $\gamma_{i,j,k}$ from $d_{i,j,k}$ to $d_1$ (by bullet (4)). Using the holonomic duality functor, it follows that there are unique arrows $\beta_{i,j,k}$ from $g_6$ to $d_{i,j,k}$ and $\delta_{i,j,k}$ from $d_1$ to $d_{i,j,k}$ (by bullet (4)). By Lemma \ref{extBD}, Lemma \ref{newExt}, and bullet (3), there are no arrows between the pairs $(d_1,g_6)$, $(d_5, d_{i,j,k})$. By Lemma \ref{noPath}, there are no arrows between $d_{i,j,k}$ and $d_{p,q,r}$ for all $(i,j,k)$ and $(p,q,r)$. We claim that there are no other arrows in $\mathcal{Q}$. Using the Fourier transform and the duality functor, this reduces to showing the following:
$$
\textnormal{Ext}^1_{\D}(S,G_6)=\textnormal{Ext}^1_{\D}(D_5,G_6)=\textnormal{Ext}^1_{\D}(D_1,S)=\textnormal{Ext}^1_{\D}(D_{i,j,k},S)=0
$$
for all $(i,j,k)$. Since $D_1$, and $D_{i,j,k}$ are not composition factors of $S_h$, the injective hull of $S$, there are no paths from their corresponding vertices in $\mathcal{Q}$ to $s$. Similarly, since $S$ and $D_5$ are not composition factors of $S_h\cdot \sqrt{h}$, the injective hull of $G_6$, there are no paths from their corresponding vertices to $g_6$. 

In remains to verify the asserted relations on the second connected component of $(\mathcal{Q},\mathcal{I})$, i.e. the connected component corresponding to the composition factors of $S_h\cdot \sqrt{h}$. Since $S_h\cdot \sqrt{h}$ is the injective hull of $G_6$, and each composition factor $G_6$, $D_{1,2,2}$, $D_{2,1,2}$, $D_{2,2,1}$, $D_1$ appears with multiplicity one, bullet (5) implies the asserted relations 
$\alpha_{i,j,k}\delta_{i,j,k}-\alpha_{p,q,r}\delta_{p,q,r}$, and $\alpha_{i,j,k}\beta_{i,j,k}$. Applying the duality functor yields the asserted relations $\gamma_{i,j,k}\beta_{i,j,k}-\gamma_{p,q,r}\beta_{p,q,r}$, and $\gamma_{i,j,k}\delta_{i,j,k}$. Next, since there are no paths in the quiver between $d_5$ and $d_{i,j,k}$ for each $(i,j,k)$, it follows from Theorem \ref{witnessWeights}, Lemma \ref{correctionLemma}, and bullet (5) that $D_{1,2,2}$, $D_{2,1,2}$, and $D_{2,2,1}$ are composition factors of the projective cover $P_{i,j,k}$ of $D_{i,j,k}$ for each $(i,j,k)$. Furthermore, since the conormal varieties $[\overline{T^{\ast}_{O_{p,q,r}}V}]$ appear in the characteristic cycle of $P_{i,j,k}$ with multiplicity one, it follows that each $D_{p,q,r}$ has multiplicity one in each $P_{i,j,k}$. Therefore, by bullet (5), we have the relations $\beta_{i,j,k}\alpha_{p,q,r}-\delta_{i,j,k}\gamma_{p,q,r}$, and $\beta_{i,j,k}\alpha_{i,j,k}$, and $\delta_{i,j,k}\gamma_{i,j,k}$. 

We only need to verify that there are no further relations on the second component of $(\mathcal{Q},\mathcal{I})$. The only possibilities  would be relations of the form $\beta_{i,j,k}\alpha_{p,q,r}$ or $\delta_{i,j,k}\gamma_{p,q,r}$ for some $(p,q,r)\neq (i,j,k)$. If $\beta_{i,j,k}\alpha_{p,q,r}$ were a relation, the relation $\beta_{i,j,k}\alpha_{p,q,r}-\delta_{i,j,k}\gamma_{p,q,r}$ would then imply that $\delta_{i,j,k}\gamma_{p,q,r}$ is a relation, which in turn would imply that there are no paths from $d_{p,q,r}$ to $d_{i,j,k}$, contradicting bullet (5) and Lemma \ref{correctionLemma}. Similarly, $\delta_{i,j,k}\gamma_{p,q,r}$ cannot be a relation. Therefore, we have found all relations, completing the proof.
\end{proof}

\section{Local Cohomology Computations}\label{Local}

We complete our analysis by computing local cohomology of some $\tn{GL}$-equivariant $\D$-modules, with support in each orbit closure. Note that for $Z\subseteq V$ a closed subvariety and $M$ a module with support contained in $Z$, we have $H^0_Z(M)=M$ and $H^j_Z(M)=0$ for $j\geq 1$. We will not discuss these cases further. We begin with the following lemma about the local cohomology of $S_h$ and $S_h\cdot \sqrt{h}$, the proof of which is analogous to the proof of \cite[Lemma 6.9]{lHorincz2018iterated}, replacing $\tn{det}$ with $h$.

\begin{lemma}\label{newLemma}
For all $j\geq 0$ and all orbits $O\neq O_6$ we have $H^j_{\overline{O}}(S_h)=H^j_{\overline{O}}(S_h\cdot \sqrt{h})=0$.
\end{lemma}

\noindent Next, we study the local cohomology of $S$. These computations are standard, but we include them for the sake of completeness. Along the way, we obtain the local cohomology of $D_1$ with support in $O_0$.

\begin{prop}\label{locS}
We have the following:
$$
H^{\bullet}_{O_0}(S)=
\begin{cases}
E & \bullet=8,\\
0 & \textnormal{otherwise},
\end{cases}\;\;\;
H^{\bullet}_{\overline{O_1}}(S)=
\begin{cases}
D_1 & \bullet=4,\\
0 & \textnormal{otherwise},
\end{cases}\;\;\;
H^{\bullet}_{\overline{O_5}}(S)=
\begin{cases}
S_h/S & \bullet=1,\\
0 & \textnormal{otherwise},
\end{cases}
$$
$$
H^{\bullet}_{O_0}(D_1)=
\begin{cases}
E & \bullet=4,\\
0 & \textnormal{otherwise},
\end{cases}\;\;\;
H^{\bullet}_{\overline{O_{i,j,k}}}(S)=
\begin{cases}
D_{i,j,k} & \bullet=3,\\
E & \bullet=5,\\
0 & \textnormal{otherwise},
\end{cases}
$$
for $(i,j,k)=(1,2,2)$, $(2,1,2)$, $(2,2,1)$.
\end{prop}

\begin{proof}
Since $S$ is a polynomial ring of dimension eight, the computation of $H^{\bullet}_{O_0}(S)$ is classical. To compute $H^{\bullet}_{\overline{O_1}}(S)$, recall that $\overline{O_1}$ is the affine cone over the Segre variety $\textnormal{Seg}(\mathbb{P}(A)\times \mathbb{P}(B)\times \mathbb{P}(C))$, a smooth variety. By \cite[Main Theorem 1.2]{switala2015lyubeznik}, the modules $H^j_{\overline{O_1}}(S)$ for $j\neq 4$ are zero. By bullet (5) in Section \ref{PrelimQuiver}, we have $\tn{Ext}^1_{\D}(S,G_6)=0$, and applying the Fourier transform we get $\tn{Ext}_{\D}^1(E,D_1)=0$ (alternatively we may use the Theorem on the Quiver Structure). By Proposition \ref{IHmod}, $H^4_{\overline{O_1}}(S)=D_1$, yielding the computation of $H^{\bullet}_{\overline{O_1}}(S)$. The spectral sequence $H^i_{O_0}(H^j_{\overline{O_1}}(S))\Rightarrow H^{i+j}_{O_0}(S)$ gives the computation of $H^{\bullet}_{O_0}(D_1)$, and the computation of $H^{\bullet}_{\overline{O_5}}(S)$ follows immediately from the \v{C}ech cohomology description of local cohomology. Finally, if we identify $\overline{O_{i,j,k}}$ with the determinantal variety of $2\times 4$ matrices of rank $\leq 1$ (see Section \ref{PrelimWit}), the computation of $H^{\bullet}_{\overline{O_{i,j,k}}}(S)$ is done in \cite[Theorem 6.1]{raicu2014locals}.
\end{proof}

We now compute the local cohomology of $D_5$ with support in each orbit closure.

\begin{prop}\label{locD5}
For all orbit closures $\overline{O}\neq V$, we have $H^{\bullet}_{\overline{O}}(\langle h^{-1}\rangle_{\D})=H^1_{\overline{O}}(\langle h^{-1}\rangle_{\D})=E$. Further, we have the following:
$$
H^{\bullet}_{O_0}(D_5)=
\begin{cases}
E & \bullet=1,7,\\
0 & \textnormal{otherwise},
\end{cases}\;\;\;
H^{\bullet}_{\overline{O_1}}(D_5)=
\begin{cases}
E & \bullet=1,\\
D_1 & \bullet=3,\\
0 & \textnormal{otherwise},
\end{cases}\;\;\;
H^{\bullet}_{\overline{O_{i,j,k}}}(D_5)=
\begin{cases}
E & \bullet=1,4,\\
D_{i,j,k} & \bullet=2,\\
0 & \textnormal{otherwise},
\end{cases}
$$
for $(i,j,k)=(1,2,2)$, $(2,1,2)$, $(2,2,1)$.
\end{prop}

\begin{proof}
The computations for $\langle h^{-1}\rangle_{\D}$ follow easily from the short exact sequence coming from the inclusion $\langle h^{-1}\rangle_{\D}\subset S_h$, using Lemma \ref{newLemma} and the long exact sequence of local cohomology. For the computations with $D_5$, consider the short exact sequence coming from the inclusion $S\subseteq \langle h^{-1}\rangle_D$. The results follow from the long exact sequence of local cohomology, using the first assertion and Proposition \ref{locS}. 
\end{proof}

Now we may finish the analysis of local cohomology of each $D_{i,j,k}$:

\begin{prop}\label{locGreenPurp}
For $(i,j,k)=(1,2,2)$, $(2,1,2)$, $(2,2,1)$ we have the following:
$$
H^{\bullet}_{O_0}(D_{i,j,k})=
\begin{cases}
E & \bullet=3,5,\\
0 & \textnormal{otherwise},
\end{cases}\;\;\;\;\;
H^{\bullet}_{\overline{O_1}}(D_{i,j,k})=H^{\bullet}_{\overline{O_{p,q,r}}}(D_{i,j,k})=
\begin{cases}
D_1 & \bullet=1,\\
E & \bullet=3,\\
0 & \textnormal{otherwise},
\end{cases}
$$
for $(i,j,k)\neq (p,q,r)$.
\end{prop}

\begin{proof}
The first assertion follows from Proposition \ref{locS}, using the spectral sequence $H^i_{O_0}(H^j_{\overline{O_{i,j,k}}}(S))\Rightarrow H^{i+j}_{O_0}(S)$. Proposition \ref{locS} and the spectral sequence $H^i_{\overline{O_1}}(H^j_{\overline{O_{i,j,k}}}(S))\Rightarrow H^{i+j}_{\overline{O_1}}(S)$ yield the computation of $H^{\bullet}_{\overline{O_1}}(D_{i,j,k})$. For the computation of $H^{\bullet}_{\overline{O_{p,q,r}}}(D_{i,j,k})$, note that $\overline{O_{i,j,k}}\cap \overline{O_{p,q,r}}=\overline{O_1}$. By Proposition \ref{locS} and the spectral sequence $H^i_{\overline{O_{p,q,r}}}(H^j_{\overline{O_{i,j,k}}}(S))\Rightarrow H^{i+j}_{\overline{O_1}}(S)$, the result follows.
\end{proof}

Finally, we investigate local cohomology of $G_6$ with various support. Recall the module $F$ which appears in the short exact sequences (\ref{sesG}) and (\ref{sesF}). We start by proving a technical lemma about local cohomology of $G_6$ and $F$:

\begin{lemma}\label{locF}
For all orbits $O\neq O_0, O_6$, we have that $H^{\bullet}_{\overline{O}}(F)=H^1_{\overline{O}}(F)=D_1$. In addition the only nonvanishing local cohomology with support in $O_0$ is: $H^5_{O_0}(F)=E$. Further, $H^0_{\overline{O_1}}(G_6)=H^1_{\overline{O_1}}(G_6)=0$. 
\end{lemma}

\begin{proof}
The first two assertions follow from the short exact sequence (\ref{sesF}), Lemma \ref{newLemma}, Proposition \ref{locS}, and the long exact sequence of local cohomology.

For the remainder of the proof, let $j$ be the open immersion $j:V\setminus \overline{O_1}\hookrightarrow V$. If we can show that $j_{\ast}j^{\ast}G_6=G_6$, then the last assertion follows from the exact sequence (\ref{locSES}) for $M=G_6$. We start by showing that $j_{\ast}j^{\ast}F=S_h\cdot\sqrt{h}$. By the first assertion and the exact sequence $(\ref{locSES})$ for $M=F$, conclude that $j_{\ast}j^{\ast}F=F\oplus D_1$ or $j_{\ast}j^{\ast}F=S_h\cdot \sqrt{h}$. By adjointness, we have $\textnormal{Hom}(D_1,j_{\ast}j^{\ast}F)=\textnormal{Hom}(j^{\ast}D_1,j^{\ast}F)=0$, where the last equality follows from the fact that $j^{\ast}D_1=0$. This proves that $j_{\ast}j^{\ast}F=S_h\cdot\sqrt{h}$.

We now show that $j_{\ast}j^{\ast}G_6=G_6$, completing the proof. Applying $H^{\bullet}_{\overline{O_1}}(-)$ to the short exact sequence $(\ref{sesG})$, we obtain a long exact sequence of local cohomology. By the first assertion and Proposition \ref{locGreenPurp}, conclude that $H^1_{\overline{O_1}}(G_6)$ is either $0$ or $D_1$. By the exact sequence (\ref{locSES}) with $M=G_6$, and Lemma \ref{extBD}, we have that $j_{\ast}j^{\ast}G_6=G_6$ or $j_{\ast}j^{\ast}=G_6\oplus D_1$. Applying $j_{\ast}j^{\ast}$ to the short exact sequence (\ref{sesG}) yields the former, as $j_{\ast}j^{\ast}F=S_h\cdot \sqrt{h}$.
\end{proof}

The previous lemma now yields:

\begin{prop}\label{locBeasier}
We have the following:
$$
H^{\bullet}_{O_0}(G_6)=
\begin{cases}
E^{\oplus 3} & \bullet=4,\\
E^{\oplus 2} & \bullet=6,\\
0 & \textnormal{otherwise},
\end{cases}\;\;\;\;
H^{\bullet}_{\overline{O_1}}(G_6)=H^{\bullet}_{\overline{O_{i,j,k}}}(G_6)=
\begin{cases}
D_1^{\oplus 2} & \bullet=2,\\
E^{\oplus 3} & \bullet=4,\\
0 & \textnormal{otherwise}.
\end{cases}
$$
In addition, $H^{\bullet}_{\overline{O_5}}(G_6)=H^1_{\overline{O_5}}(G_6)=(S_h\cdot\sqrt{h})/G_6$.
\end{prop}

\begin{proof}
Consider the long exact sequences of local cohomology obtained by applying $H^{\bullet}_{\overline{O_1}}(-)$ or $H^{\bullet}_{\overline{O_{i,j,k}}}(-)$ to the short exact sequence (\ref{sesG}). By Proposition \ref{locGreenPurp} and Lemma \ref{locF}, we obtain the second assertion. The first assertion now follows from the spectral sequence $H^i_{O_0}(H^j_{\overline{O_1}}(G_6))\Rightarrow H^{i+j}_{O_0}(G_6)$. The final assertion follows from the \v{C}ech cohomology description of local cohomology.
\end{proof}

This completes our study of the local cohomology of the simple objects. We remark about computing local cohomology of $S_h/S$ and $(S_h\cdot \sqrt{h})/G_6$ with support in orbit closures. Using the short exact sequences coming from the inclusions $S\subset S_h$ and $G_6\subset S_h\cdot \sqrt{h}$ and Lemma \ref{newLemma}, we immediately obtain for all orbits $O\neq O_6$ that $H^{\bullet}_{\overline{O}}(S_h/S)=H^{\bullet+1}_{\overline{O}}(S)$, and $H^{\bullet}_{\overline{O}}((S_h\cdot \sqrt{h})/G_6)=H^{\bullet+1}_{\overline{O}}(G_6)$. This allows one to compute any iteration of local cohomology $H^{i_1}_{\overline{O^1}}(\cdots (H^{i_t}_{\overline{O^t}}(M)\cdots)$ of any simple object $M$ with support in orbit closures $\overline{O^i}$.

\section*{Acknowledgments}
The author is very grateful to Claudiu Raicu for his guidance while this work was done. We thank Andr\'{a}s C. L\H{o}rincz and Kostya Timchenko for helpful conversations. We also thank Andr\'{a}s for pointing out an issue with the relations of the quiver in an earlier version of this article, and for suggesting how to correct them.  The author was supported by the NSF Graduate Research Fellowship under Grant No. DGE-1313583.

\bibliographystyle{alpha}
\bibliography{mybib}

\Addresses

\end{document}